%% file: girth-arXiv.tex
\documentclass[a4paper,UKenglish,cleveref, autoref, thm-restate]{lipics-v2021}
\title{Matching Cuts in Graphs of High Girth and $H$-Free Graphs}

\author{Carl Feghali}{University of Lyon, EnsL, CNRS, LIP, F-69342, Lyon Cedex 07, France}{carl.feghali@ens-lyon.fr}{0000-0001-6727-7213}{}
\author{Felicia Lucke}{Department of Informatics, University of Fribourg, Fribourg, Switzerland}{felicia.lucke@unifr.ch}{https://orcid.org/0000-0002-9860-2928}{}
\author{Daniël Paulusma}{Department of Computer Science, Durham University, Durham, UK}{daniel.paulusma@durham.ac.uk}{https://orcid.org/0000-0001-5945-9287}{}
\author{Bernard Ries}{Department of Informatics, University of Fribourg, Fribourg, Switzerland}{bernard.ries@unifr.ch}{https://orcid.org/0000-0003-4395-5547}{}

\authorrunning{C. Feghali and F. Lucke and D. Paulusma and B.Ries} 
\Copyright{Carl Feghali and Felicia Lucke and Dani\"el Paulusma and Bernard Ries}
\ccsdesc[100]{Mathematics of computing~Graph algorithms}
\keywords{matching cut; perfect matching; girth; $H$-free graph} 

\acknowledgements{We thank Hoang-Oanh Le for a significant simplification of our original proof of Theorem~\ref{thm:girth_MC}, which we simplified a bit further. We thank Van~Bang Le for observing the bound on the maximum degree in Theorem~\ref{thm:girth_MC}, solving an Open Problem Garden question.}

\nolinenumbers 
\hideLIPIcs

\usepackage{boxedminipage,tikz}
\usepackage{subcaption}
\usepackage{nicefrac}
\usetikzlibrary{decorations.pathreplacing, calligraphy}
\usetikzlibrary{decorations.pathmorphing}
\usetikzlibrary{arrows,automata}
\usetikzlibrary{arrows.meta}
\definecolor{nicered}{RGB}{204,0,0}
\definecolor{lightblue}{RGB}{153,204,255}
\definecolor{nicegreen}{RGB}{0,153,0}
 \tikzstyle{vertex}=[thin,circle,inner sep=0.cm, minimum size=1.7mm, fill=black, draw=black]
  \tikzstyle{svertex}=[thin,circle,inner sep=0.cm, minimum size=1.3mm, fill=black, draw=black]
 \tikzstyle{bvertex}=[thin,circle,inner sep=0.cm, minimum size=1.7mm, fill=lightblue, draw=lightblue]
 \tikzstyle{rvertex}=[thin,circle,inner sep=0.cm, minimum size=1.7mm, fill=nicered,draw=nicered]
 \tikzstyle{hedge}=[thick, draw = gray]
  \tikzstyle{edge}=[thick, draw = gray]
 \tikzstyle{medge}=[ultra thick, draw = black]
 \tikzstyle{gedge}=[ultra thick, draw = nicered]
 \tikzstyle{pedge}=[ultra thick, draw=lightblue]
  \tikzstyle{rededge}=[thick, draw = nicered]
 \tikzstyle{bluedge}=[thick, draw = lightblue]
 \tikzstyle{greenedge}=[thick, draw = nicegreen]
 \tikzstyle{edge}=[thick, draw = gray]
 \tikzstyle{br} = [decorate, ultra thick, decoration = {calligraphic brace}]
  \tikzstyle{wiggly} = [decorate, decoration = snake, thick, draw = gray,]
  
  \pgfdeclarelayer{background}
  \pgfsetlayers{background,main}

\newcommand{\NP}{{\sf NP}}

\newcommand{\ssi}{\subseteq_i}
\newcommand{\si}{\supseteq_i}

\DeclareMathOperator*{\dist}{dist}

\newtheorem{claim1}{Claim}[theorem]
\newtheorem{open}{Open Problem}

\begin{document}

\maketitle

\begin{abstract}
The {\sc (Perfect) Matching Cut} problem is to decide if a connected graph has a (perfect) matching that is also an edge cut. The {\sc Disconnected Perfect Matching} problem is to decide if a connected graph has a perfect matching that contains a matching cut. Both {\sc Matching Cut} and {\sc Disconnected Perfect Matching} are  \NP-complete for planar graphs of girth~$5$, whereas {\sc Perfect Matching Cut} is known to be \NP-complete even for subcubic bipartite graphs of arbitrarily large fixed girth. We prove that {\sc Matching Cut} and  {\sc Disconnected Perfect Matching} are also \NP-complete for bipartite graphs of arbitrarily large fixed girth and bounded maximum degree. Our result for {\sc Matching Cut} resolves a 20-year old open problem. We also show that the more general problem {\sc $d$-Cut}, for every fixed $d\geq 1$, is \NP-complete for bipartite graphs of arbitrarily large fixed girth and bounded maximum degree. Furthermore, we show that {\sc Matching Cut}, {\sc Perfect Matching Cut} and {\sc Disconnected Perfect Matching} are \NP-complete for $H$-free graphs whenever $H$ contains a connected component with two vertices of degree at least~$3$. Afterwards, we update the state-of-the-art summaries for $H$-free graphs and compare them with each other, and with a known and full classification of the {\sc Maximum Matching Cut} problem,  which is to determine a largest matching cut of a graph~$G$. Finally, by combining existing results, we obtain a complete complexity classification of {\sc Perfect Matching Cut} for ${\cal H}$-subgraph-free graphs where ${\cal H}$ is any finite set of graphs.
\end{abstract}

\section{Introduction}\label{s-intro}

We consider classic graph problems for finding certain edge cuts, which 
have in common that their edges must form a matching. In order to explain this, let $G=(V,E)$ be a connected graph. A set $M\subseteq E$ is a {\it matching} of $G$ if no two edges in $M$ share an end-vertex; $M$ is {\it perfect} if every vertex of $G$ is incident to an edge of $M$. A set $M\subseteq E$ is an {\it edge cut} of $G$ if $V$ can be partitioned into two sets $B$ and $R$, such that $M$ consists of all the edges with one end-vertex in $B$ and the other one in $R$. We say that $M$ is a {\it (perfect) matching cut} of $G$ if $M$ is a (perfect) matching that is also an edge cut. We refer to Figure~\ref{f-examples} for some examples.

Graphs with matching cuts were introduced in 1970 by Graham~\cite{Gr70} as {\it decomposable} graphs. Matching cuts have applications in number theory~\cite{Gr70}, graph drawing~\cite{PP01}, graph homomorphisms~\cite{GPS12}, edge labelings~\cite{ACGH12} and ILFI networks~\cite{FP82}. 
Moreover, a connected graph with no vertex of degree~$1$ has a matching cut if and only if its line graph has a vertex cut that is an independent set ({\it stable cut set}).
As such, (perfect) matching cuts are well studied in the literature. 

Instead of considering perfect matchings that {\it are} edge cuts, we can also consider perfect matchings in graphs that {\it contain} edge cuts. Such perfect matchings are called {\it disconnected perfect matchings}; see Figure~\ref{f-examples} again. Note that every perfect matching cut is a disconnected perfect matching. However, there exist connected graphs, like the cycle $C_6$ on six vertices, that have a disconnected perfect matching (and thus a matching cut) but no perfect matching cut. There also exist connected graphs, like the path $P_3$ on three vertices, that have a matching cut, but no disconnected perfect matching (and thus no perfect matching cut either).

The problems {\sc Matching Cut}, {\sc Disconnected Perfect Matching} and {\sc Perfect Matching} are to decide if a connected graph has a matching cut, disconnected perfect matching or perfect matching cut, respectively. As explained below, all three problems are \NP-complete, and have been extensively studied for special graph classes.

\begin{figure}[t]
\centering
\begin{tikzpicture}
\tikzstyle{cutedge}=[black, ultra thick]
\tikzstyle{matchedge}=[black, ultra thick, dotted]

\begin{scope}[shift={(10,0)}, scale=0.75]
\node[rvertex](p1) at (0,0){};
\node[bvertex](p2) at (1,0){};
\node[bvertex](p3) at (2,0){};
\node[rvertex](p4) at (3,0){};
\node[rvertex](p5) at (4,0){};
\node[bvertex](p6) at (5,0){};

\draw[cutedge](p1)--(p2);
\draw[hedge](p2)--(p3);
\draw[cutedge](p3)--(p4);
\draw[hedge](p4)--(p5);
\draw[cutedge](p5)--(p6);
\end{scope}

\begin{scope}[shift={(5,0)}, scale=0.75]
\node[rvertex](p1) at (0,0){};
\node[bvertex](p2) at (1,0){};
\node[bvertex](p3) at (2,0){};
\node[bvertex](p4) at (3,0){};
\node[bvertex](p5) at (4,0){};
\node[rvertex](p6) at (5,0){};

\draw[cutedge](p1)--(p2);
\draw[hedge](p2)--(p3);
\draw[hedge](p3)--(p4);
\draw[hedge](p4)--(p5);
\draw[cutedge](p5)--(p6);
\end{scope}

\begin{scope}[shift={(0,0)}, scale=0.75]
\node[rvertex](p1) at (0,0){};
\node[bvertex](p2) at (1,0){};
\node[bvertex](p3) at (2,0){};
\node[bvertex](p4) at (3,0){};
\node[rvertex](p5) at (4,0){};
\node[rvertex](p6) at (5,0){};

\draw[cutedge](p1)--(p2);
\draw[hedge](p2)--(p3);
\draw[hedge](p3)--(p4);
\draw[cutedge](p4)--(p5);
\draw[hedge](p5)--(p6);
\end{scope}
\end{tikzpicture}
\caption{The graph $P_6$ from~\cite{LPR22b} with a matching cut that is not contained in a disconnected perfect matching (left), a matching cut that is properly contained in a disconnected perfect matching (middle) and a perfect matching cut (right). In each figure, thick edges denote matching cut edges.}\label{f-examples}
\end{figure}
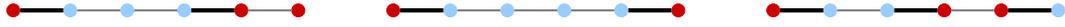

\medskip
\noindent
{\bf Our Focus.}
 The {\it girth} of a graph that is not a forest is the number of edges of a shortest cycle in it; a forest has infinite girth. In 2003, Bonsma~\cite{Bo09} asked if {\sc Matching Cut} is still \NP-complete for graphs of large girth and proved that this is indeed the case for planar graphs of girth~$5$. In the 2009 journal version of~\cite{Bo09}, Bonsma showed that every connected planar graph of girth at least~$6$ has a matching cut. Hence, the complexity status of {\sc Matching Cut} for graphs of large girth remained unknown and was regularly posed as an open problem~\cite{CHLLP21,LL19,LT21,LPR22a}.

Bouquet and Picouleau~\cite{BP} proved that {\sc Disconnected Perfect Matching} is \NP-complete for planar graphs of girth~$g=5$ and left the cases where $g\geq 6$ open.
In contrast, Le and Telle~\cite{LT21} proved that for every $g\geq 3$, {\sc Perfect Matching Cut} is \NP-complete for subcubic bipartite graphs of girth at least $g$ (a graph is {\it subcubic} if it has maximum degree at most~$3$).
We focus on the two remaining open problems:

\medskip
\noindent
{\it What is the complexity of {\sc Matching Cut} and {\sc Disconnected Perfect Matching} for graphs of large girth?}

\medskip
\noindent
A challenging task is to find gadgets with no edges that are subdivided twice; such gadgets always have a matching cut (take the two subdivision vertices on one side and all other vertices on the other) and cannot be used in any hardness reduction. 

\subsection{Other Relevant Known Results}\label{s-known}

We restrict ourselves to {\it hereditary} graph classes, that is, classes of graphs closed under vertex deletion. A class of graphs~${\cal G}$ is hereditary if and only if the graphs in ${\cal G}$ are {\it ${\cal F}_{\cal G}$-free} for some unique set ${\cal F}_{\cal G}$, that is, they do not contain any graph from ${\cal F}_{\cal G}$ as an induced subgraph. For a {\it systematic} study, one may start with {\it $H$-free} graphs (so where ${\cal F}_{\cal G}$ has a single graph~$H$).

\medskip
\noindent
{\bf Matching Cuts.}
Chv\'atal~\cite{Ch84} proved that {\sc Matching Cut} is \NP-complete even for $K_{1,4}$-free graphs of maximum degree~$4$ (the graph $K_{1,r}$ is the $(r+1)$-vertex star); see~\cite{PP01} for an alternative hardness proof. In contrast, Chv\'atal~\cite{Ch84} also proved that {\sc Matching Cut} is polynomial-time solvable for graphs of maximum degree at most~$3$, whereas Bonsma~\cite{Bo09} proved the same for $K_{1,3}$-free graphs, thereby generalizing a known result of Moshi~\cite{Mo89} for line graphs, and also for $P_4$-free graphs; the latter result was extended to $P_5$-free graphs in~\cite{Fe23} and to $P_6$-free graphs in~\cite{LPR22a}.
Kratsch and Le~\cite{KL16} proved polynomial-time solvability for $(K_{1,4},K_{1,4}+e)$-free graphs.
It is also known that if {\sc Matching Cut} is polynomial-time solvable for $H$-free graphs for some graph~$H$, then it is so for $(H+P_3)$-free graphs~\cite{LPR22a}; for two vertex-disjoint graphs $G_1$ and $G_2$, we write $G_1+G_2=(V(G_1)\cup V(G_2),E(G_1)\cup E(G_2))$. 

Moshi~\cite{Mo89} proved that {\sc Matching Cut} is \NP-complete even for bipartite graphs where the vertices in one set of the bipartition all have degree exactly~$2$. Consequently,
{\sc Matching Cut} is \NP-complete for $(H_1^*,H_3^*,H_5^*,\ldots)$-free bipartite graphs, where $H_1^*=H^*$ denotes the graph that looks like the letter ``H'', and for $i\geq 2$, the graph $H_i^*$ is the graph obtained from $H^*=H_1^*$ by subdividing the middle edge of $H_1^*$ exactly $i-1$ times; see also Figure~\ref{fig-Hstar}.

Le and Randerath~\cite{LR03} proved that {\sc Matching Cut} is \NP-complete for $K_{1,5}$-free bipartite graphs and thus for $C_s$-free graphs if $s$ is odd. Recall that Bonsma~\cite{Bo09} proved that {\sc Matching Cut} is \NP-complete for planar graphs of girth~$5$, and thus for $(C_3,C_4)$-free graphs. By using the graph transformation of Moshi~\cite{Mo89}, {\sc Matching Cut} is \NP-complete for $C_s$-free graphs also if $s$ is even and at least~$6$~\cite{LPR22a}.
 In~\cite{LPR22b} it was shown that {\sc Matching Cut} is \NP-complete even for $(3P_5,P_{15})$-free graphs, strengthening a result of~\cite{Fe23}. 
 Afterwards, Le and Le~\cite{LL23} proved that {\sc Matching Cut} is \NP-complete even for $(3P_6,2P_7,P_{14})$-free graphs.
 
\begin{figure}[t]
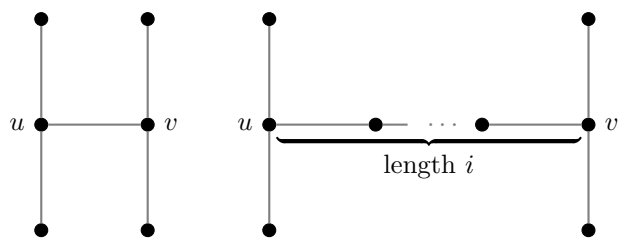

\centering
\include{Hstar}
\caption{The graphs $H^* = H_1^*$ (left) and $H_i^*$ (right).}\label{fig-Hstar} 
\end{figure}

We refer to~\cite{BJ08,LL19,LPR22a} for results for non-hereditary graph classes, to~\cite{AKK22,GKKL22,GS21,KKL20,KL16} for parameterized complexity results and exact algorithms, to~\cite{AS21,GS21} for a generalization of  {\sc Matching Cut} to {\sc $d$-Cut} 
(where instead of at most one neighbour we allow each vertex to have at most $d$ neighbours across the cut)
and to~\cite{CHLLP21} for a  comprehensive overview.

\medskip
\noindent
{\bf Disconnected Perfect Matchings.}
The {\sc Disconnected Perfect Matching} problem was introduced by Bouquet and Picouleau~\cite{BP}. They used a different name but we adapt the name of Le and Telle~\cite{LT21} to avoid confusion with {\sc Perfect Matching Cut}.  Bouquet and Picouleau~\cite{BP} showed that {\sc Disconnected Perfect Matching} is, among others, polynomial-time solvable for $K_{1,3}$-free graphs and $P_5$-free graphs, but \NP-complete for $K_{1,4}$-free planar graphs, planar graphs of girth~$5$ and for bipartite graphs, and thus for $C_s$-free graphs for every odd $s$. 
Le and Le~\cite{LL23} proved that {\sc Disconnected Perfect Matching} is \NP-complete even for $(3P_6,2P_7,P_{14})$-free graphs, which improved the previous \NP-completeness result of~\cite{LPR22b} for $(3P_7,P_{19})$-free graphs. 

\medskip
\noindent
{\bf Perfect Matching Cuts.}
The {\sc Perfect Matching Cut} problem was first shown to be \NP-complete by Heggernes and Telle~\cite{HT98}.
Le and Telle~\cite{LT21} proved, besides \NP-completeness for subcubic bipartite graphs of girth at least $g$ (for any $g\geq 3$), 
that {\sc Perfect Matching Cut} is polynomial-time solvable for chordal graphs and for $S_{1,2,2}$-free graphs; the graph $S_{1,2,2}$ is obtained by subdividing two of the edges of the claw $K_{1,3}$ exactly once. In~\cite{LPR22b}, it was shown that {\sc Perfect Matching Cut} is polynomial-time solvable for $P_6$-free graphs, and moreover for $(H+P_4)$-free graphs if it is polynomial-time solvable for $H$-free graphs. 

Even more recently, two new results were shown.
Le and Le~\cite{LL23} proved that {\sc Perfect Matching Cut} is \NP-complete even for $(3P_6,2P_7,P_{14})$-free graphs (using the same construction as for {\sc Matching Cut} and {\sc Disconnected Perfect Matching}), whereas Bonnet, Chakraborty and Duron~\cite{BCD23} proved that  {\sc Perfect Matching Cut} is \NP-complete for $3$-connected cubic bipartite planar graphs.

\subsection{Our Results}\label{s-her}

As our main results, we solve, in Section~\ref{s-girth}, the aforementioned two open problems in the literature~\cite{Bo09,BP,CHLLP21,LL19,LT21,LPR22a}
by showing that for every $g\geq 3$, {\sc Matching Cut} is \NP-complete for bipartite graphs of girth at least~$g$ and maximum degree at most~$60$ and {\sc Disconnected Perfect Matching} is \NP-complete for bipartite graphs of girth at least~$g$ and maximum degree at most~$74$. 
Note that our first result answers, in the negative, a question~\cite{open} from the {\it Open Problem Garden}:

\medskip
\noindent
{\it ``For every $d$ does there exists a $g$ such that every graph with average degree smaller than $d$ and girth at least $g$ has a matching-cut?''}

\medskip
\noindent
As an immediate consequence of our second result, we have that {\sc Disconnected Perfect Matching} is \NP-complete for $C_s$-free graphs for even~$s$ (as mentioned above, previously this was only known for odd $s$~\cite{BP}).
To overcome the obstacle that connected graphs with a $2$-subdivided edge have a matching cut, we use results from the theory of expander graphs.
The proof of our result on {\sc Matching Cut} is surprisingly short. 
Moreover, we can extend it in a straightforward way to show the following. For every $d\geq 2$ and $g\geq 3$, there 
exists a function~$f(d)$ that only depends on $d$, such that even {\sc $d$-Cut} is \NP-complete for bipartite graphs of girth at least~$g$ and maximum degree at most $f(d)$; recall that {\sc $1$-Cut} is the {\sc Matching Cut} problem.

In Section~\ref{s-girth}, we also highlight an implicit result of Le and Telle~\cite{LT21} for {\sc Perfect Matching Cut}. In their \NP-hardness proof for subcubic bipartite graphs of high girth, they showed that a graph~$G$ has a {\it perfect} matching cut if and only if the graph obtained from~$G$ by subdividing some edge four times has a {\it perfect} matching cut. Since graphs with a $2$-subdivided edge have a matching cut, this property shows a fundamental difference between matching cuts and perfect matching cuts.

For $i\geq 2$, recall from Figure~\ref{fig-Hstar} that $H_i^*$ is the graph obtained from $H^*=H_1^*$ by subdividing the middle edge of $H_1^*$ exactly $i-1$ times.
Recall also that a result of Moshi~\cite{Mo89} implies that {\sc Matching Cut} is \NP-complete for $(H_1^*,H_3^*,H_5^*,\ldots)$-free bipartite graphs.
In Section~\ref{s-h}, we extend this result by
proving that for every $i\geq 1$, {\sc Matching Cut} and {\sc Disconnected Perfect Matching}  are \NP-complete for  $(H_1^*,\ldots,H_i^*)$-free graphs. 
We obtain these results by replacing the gadget of Moshi~\cite{Mo89} with more advanced graph transformations.
In Section~\ref{s-h}, we also make explicit that the construction of Le and Telle~\cite{LT21} 
for subcubic bipartite graphs of girth at least~$g$ gives in fact \NP-completeness of {\sc Perfect Matching Cut} for $(H_1^*,\ldots,H_i^*)$-free subcubic bipartite graphs of girth at least~$g$, for any $g\geq 3$.
Hence, all three problems are \NP-complete for $H$-free graphs whenever $H$ has a connected component with two vertices of degree at least~$3$.

In Section~\ref{s-con}, we 
combine our new results with the known results. We 
update the state-of-the-art summaries for the three problems on $H$-free graphs;
basically, for all three problems, we only need to consider cases where $H$ is a linear forest. 
Apart from comparing the (partial) classification of the three problems with each other, we also compare them with a recent complete classification of the {\sc Maximum Matching Cut} problem for $H$-free graphs~\cite{LPR23}. 
This problem is to determine a matching cut in a graph~$G$ with a maximum number of edges (or output that no matching cut exists in $G$).
Finally, we show how the results for {\sc Perfect Matching Cut} lead to a full classification of {\sc Perfect Matching Cut} on ${\cal H}$-subgraph-free graphs, for every finite set of graphs~${\cal H}$.

\section{Preliminaries}\label{s-pre}

We only consider finite, undirected graphs without multiple edges and self-loops. 
Let $G=(V,E)$ be a graph. For $u\in V$, the set $N(u)=\{v \in V\; |\; uv\in E\}$ is the {\it neighbourhood} of $u$ in $G$, where $|N(u)|$ is the {\it degree} of $u$. 
For an integer~$p\geq 0$, $G$ is $p$-regular if every $u\in V$ has degree~$p$.
Let $S\subseteq V$. The {\it neighbourhood} of $S$ is the set $N(S)=\bigcup_{u\in S}N(u)\setminus S$. 
The graph $G[S]$ is the subgraph of~$G$ {\it induced} by $S\subseteq V$, that is, $G[S]$ is the graph obtained from~$G$ after deleting the vertices not in $S$. We write $G-S=G[V\setminus S]$.
Let $u,v\in V$. The {\it distance} between $u$ and $v$ in~$G$ is the {\it length} (number of edges) of a shortest path between $u$ and $v$ in $G$. 
The subdivision of an edge $e=uv$ of $G$ replaces $e$ by a new vertex $w$ and edges $uw$ and $wv$.

We will now define some useful colouring terminology for matching cuts used in other papers as well (see e.g.~\cite{LPR22b}). In the remainder of this section, we let $G=(V,E)$ be a connected graph.
A {\it red-blue colouring} of $G$ colours every vertex of $G$ either red or blue. If every vertex of some set $S\subseteq V$ has the same colour (red or blue), then $S$ is said to be {\it monochromatic}. We also say that $G[S]$ is {\it monochromatic}.
A red-blue colouring of $G$ is {\it valid} if the following holds:

\begin{enumerate}
\item every blue vertex has at most one red neighbour;
\item every red vertex has at most one blue neighbour; and
\item both colours red and blue are used at least once.  
\end{enumerate}

\noindent
If a red vertex $u$ in $G$ has a blue neighbour $v$, then $u$ and $v$ are said to be \textit{matched}. 
See Figure~\ref{f-examples} for three examples of valid red-blue colourings of the $P_6$.

For a valid red-blue colouring of $G$, we let $R$ be the {\it red} set consisting of all vertices coloured red and $B$ be the {\it blue} set consisting of all vertices coloured blue. Note that $V=R\cup B$. The {\it red interface} is the set $R'\subseteq R$ consisting of all vertices in $R$ with a (unique) blue neighbour, and the {\it blue interface} is the set $B'\subseteq B$ consisting of all vertices in~$B$ with a (unique) red neighbour in $R$. 

A red-blue colouring of~$G$ is {\it perfect} if it is valid and moreover $R'=R$ and $B'=B$; see Figure~\ref{f-examples} (middle) for an example of a perfect red-blue colouring (of the $P_6$). A red-blue colouring of $G$ is {\it perfect-extendable} if it is valid and $G[R\setminus R']$ and $G[B \setminus B']$ both contain a perfect matching; see Figure~\ref{f-examples} (right) for an example of a perfect-extendable red-blue colouring (of the $P_6$).
 In other words, the matching defined by the edges with one end-vertex in $R'$ and the other one in $B'$ can be extended to a perfect matching in $G$ or, equivalently, is contained in a perfect matching in $G$.

We now make the following straightforward observation (see also e.g.~\cite{LPR22b}).

\begin{observation}\label{o} Let $G$ be a connected graph. The following three statements hold:
\begin{itemize}
\item [(i)] $G$ has a matching cut if and only if $G$ has a valid red-blue colouring;
\item [(ii)] $G$ has a disconnected perfect matching if and only if $G$ has a perfect-extendable red-blue colouring;
\item [(iii)] $G$ has a perfect matching cut if and only if $G$ has a perfect red-blue colouring.
\end{itemize}
\end{observation}

\noindent
Finally, we formally define the notion of a $d$-cut. For an integer~$d\geq 1$ and a connected graph $G=(V,E)$, a set $M\subseteq E$ is a {\it $d$-cut} of $G$ if $V$ can be partitioned into two sets $B$ and $R$, such that the following two conditions hold:

\begin{itemize}
\item [(i)] $M$ consists of all the edges with one end-vertex in $B$ and the other one in $R$; and
\item [(ii)] every vertex in $B$ has at most $d$ neighbours in $R$, and vice versa.\\[-10pt]
\end{itemize}

\noindent
Recall that the corresponding {\sc $d$-Cut} problem is to decide if a connected graph has a $d$-cut. Hence, {\sc $1$-Cut} and {\sc Matching Cut} are the same problems.

\section{Hardness for Arbitrary Given Girth}\label{s-girth}

In this section, we will show that {\sc Matching Cut} (Section~\ref{s-31}) and {\sc Disconnected Perfect Matching} (Section~\ref{s-32}) remain \NP-complete even for graphs of high girth and bounded maximum degree; recall that previously both problems were known to be \NP-complete for (planar) graphs of girth at least~$g$, only for $g\leq 5$~\cite{Bo09,BP}. 
We also briefly discuss how an implicit observation of Le and Telle~\cite{LT21} gives
a simple, alternative proof for showing \NP-completeness for {\sc Perfect Matching Cut} (Section~\ref{s-33}) for graphs of high girth.

We need the following notions.
The {\it edge expansion} $h(G)$ of a graph $G=(V,E)$ on $n$~vertices is defined as
$$h(G) = \min_{1 \leq |S| \leq \frac{n}{2}} \frac{|\partial S|}{|S|},$$
where $\partial S := \{uv \in E\; |\; u \in S, v \in V \setminus S\}$
is the set of all edges of $G$ with one end-vertex in~$S$ and one end-vertex outside $S$.

A connected graph~$G$  is said to be \emph{(matching-)immune} if $G$ admits no matching cut. We generalize this notion as follows.
For an integer $d\geq 1$, we say that a connected graph $G$ is {\it $d$-immune} if $G$ admits no $d$-cut, so being $1$-immune is the same as being immune.
We make the following observation.

\begin{observation}\label{o-mi}
For every integer $d\geq 1$, every connected graph $G$ with $h(G)> d$ is $d$-immune.
\end{observation}

\begin{proof}
Let $d\geq 1$, and let $G$ be a connected graph with $h(G)> d$. For every set $S \subseteq V$ with 
$1 \leq |S| \leq \frac{n}{2}$, the average number of neighbours the vertices of $S$ have in $V\setminus S$ is strictly larger than~$d$. Hence, $G$ cannot have a $d$-cut. In other words, $G$ is $d$-immune.
\end{proof}

\subsection{Matching Cut and d-Cut}\label{s-31}

In order to prove our hardness results for {\sc Matching Cut} and {\sc $d$-Cut} for bipartite graphs of high girth and bounded maximum degree, we use known results on expander graphs and number theory.

Two integers $a$ and $b$ are \emph{coprime}, if they do not have a common divisor greater than~$1$. The following result is well known.

\begin{theorem}[\cite{LD37orig}]\label{thm:dirichlet}
For two positive, coprime integers $a$ and $b$, the sequence $a +bk$, for $ k \in \mathbb{N}$ contains infinitely many primes.
\end{theorem}

\noindent
For an integer $a$ and a prime $p$, the \textit{Legendre symbol} is defined as $\left(\frac{a}{p}\right) \equiv a^{\frac{p-1}{2}} \bmod p$. It is well known (see, for example,~\cite{Sh08}) that for any integer~$a$ and prime~$p$, it holds that $\left(\frac{a}{p}\right) \equiv 0\bmod p$, or $1 \bmod p$, or $(p-1)\bmod p$, and with slight abuse of notation one denotes these integers by $0$, $1$ and $-1$, respectively.
 
We use Theorem~\ref{thm:dirichlet} to prove the following lemma.

\begin{lemma}\label{l-suitableprimes}
There are infinitely many primes $q$, such that
\begin{itemize}
\item[a)] $ q > 13$, 
\item[b)] $q \equiv 1 \bmod 4$, and 
\item[c)] the Legendre symbol $ \left(\frac{q}{13}\right) = -1$.
\end{itemize}
\end{lemma}

\begin{proof}
We  apply Theorem~\ref{thm:dirichlet} for $a = 5$ and $b = 52$, which are coprime. This yields that the sequence $5+52k$, for $k \in \mathbb{N}$, contains infinitely many primes, and consequently,  infinitely many primes satisfying a). Since for any $k \in \mathbb{N}$, we have $5+52k \equiv 1 \bmod 4$, condition~b) is also satisfied. 

Finally, we will show that the Legendre symbol $\left(\frac{5 + 52k}{13}\right)= -1$. First we note that 
if $a \equiv  b \bmod p$, then $\left(\frac{a}{p}\right) = \left(\frac{b}{p}\right)$ (see~\cite{Sh08}). Hence, since $5 \equiv 5+52k \bmod 13$, we have $\left(\frac{5}{13}\right) = \left(\frac{5 + 52k}{13}\right)$. We now deduce that $\left(\frac{5}{13}\right) \equiv 5^{\frac{13-1}{2}} \bmod 13 \equiv 12 \bmod13 \equiv -1 \bmod 13$. This proves condition~c).
\end{proof}

\noindent
Lemma~\ref{lem:girth_matching_immune} uses known results from the theory of expander graphs. We need the fact that the graph in the statement of Lemma~\ref{lem:girth_matching_immune} has a perfect matching only in Section~\ref{s-32}. 

\begin{lemma} \label{lem:girth_matching_immune}
For every $g \geq 3$, there is an immune $14$-regular bipartite graph with girth at least $g$ that contains a perfect matching.
\end{lemma}

\begin{proof}
It follows from a construction of Lubotzky, Phillips and Sarnak~\cite{LPS88} based on Caley graphs that for every two primes $p$ and $q$ with the following four properties

\begin{itemize}
\item $q > p$, 
\item $p \equiv 1 \bmod 4$,
\item $q \equiv 1 \bmod 4$, and 
\item $\left(\frac{p}{q}\right) = -1$,
\end{itemize} 

\noindent
there exists a $(p+1)$-regular bipartite graph $G$ with $\lambda_2(G) \leq 2 \sqrt{p}$ and girth at least $4\log_p q-\log_p 4$; here, $\lambda_2(G)$ denotes the second largest eigenvalue of the adjacency matrix of $G$.

We set $p=13$, so $p\equiv 1 \bmod 4$. We now combine the above with Lemma~\ref{l-suitableprimes} to find that there exist infinitely many primes $q$, such that there exists a $14$-regular bipartite graph~$G_q$ with $\lambda_2(G_q) \leq 2 \sqrt{13}$ and girth at least $4\log_{13} q-\log_{13} 4$. Moreover, Dodziuk \cite{Do84} and, independently, Alon and Milman~\cite{AM85} showed that for every integer~$\ell\geq 1$, every $\ell$-regular graph~$G$ satisfies $h(G) \geq \frac{1}{2}(\ell - \lambda_2(G))$. This means that
 $$h(G_q) \geq \frac{1}{2}(14 - \lambda_2(G_q))\geq \frac{1}{2}(14 -2 \sqrt{13}) \approx 3.39 >1.$$
 Hence, by applying Observation~\ref{o-mi}, we find that $G_q$ is immune.
 
By taking $q$ sufficiently large, we conclude that for any $g\geq 3$, there exists an immune $14$-regular bipartite graph $G$ with girth at least $g$.
Finally, as $G$ is bipartite and regular, we find that $G$ has a perfect matching (due to Hall's Marriage Theorem). 
\end{proof}

\noindent
{\bf Remark~1.}
The arguments that we used to prove Lemma~\ref{lem:girth_matching_immune} do not allow us to set $p=13$ to a smaller value. We need $p$ to be prime, and moreover, it must hold that $p\equiv 1\bmod 4$. Hence, the only  alternative value for $p$ that is smaller than $13$ would be $p=5$. However, $p=5$ yields 
$h(G_q) \geq \frac{1}{2}(6 -2 \sqrt{5}) \approx 0.76$, so we cannot conclude from this that $G_q$ is immune.

\medskip
\noindent
We use Lemma~\ref{lem:girth_matching_immune} in the proof of our first main result.

\begin{figure}[t]
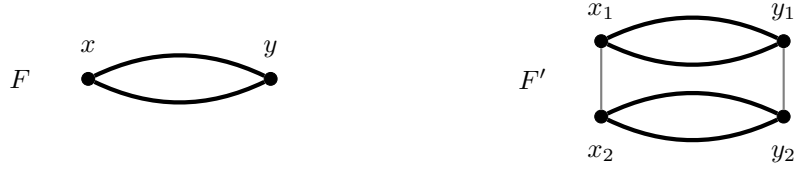

\begin{center}
\include{dpm_girth}
\end{center}
\caption{The graph $F$ with designated vertices $x$ and $y$ at distance at least $g$ (left) and the graph~$F'$ (right) from the proof of Theorem \ref{thm:girth_MC}.}\label{fig:bad}
\end{figure}

\begin{theorem}\label{thm:girth_MC}
For every integer $g \geq 3$, {\sc Matching Cut} is \NP-complete for bipartite graphs of girth at least~$g$ and maximum degree at most~$60$.
\end{theorem}

\begin{proof}
Let $g\geq 3$. As the class of graphs of girth at least $g+1$ is a subclass of the class of graphs of girth at least~$g$, we may assume without loss of generality that $g$ is divisible by~$2$ but not by~$4$, so $\frac{g}{2}$ is odd. We reduce from {\sc Matching Cut}. Recall that {\sc Matching Cut} is \NP-complete for graphs of degree at most~$4$~\cite{Ch84}.  Let $G$ be a connected graph of maximum degree at most~$4$. From $G$ we construct a graph $G'$ with the required properties, but first we define an auxiliary graph $F'$.

By Lemma~\ref{lem:girth_matching_immune} there exists an immune $14$-regular bipartite graph $F$ with girth at least~$g$. 
Note that $F$ has constant size, so we can find $F$ in constant time.

Let $x$ and $y$ be two vertices of distance~$\frac{g}{2}$ in $F$. As $\frac{g}{2}$ is odd, $x$ and $y$ belong to opposite bipartition classes of $F$.
We take two copies $F_1$ and $F_2$ of $F$, and we add an edge between the two copies $x_1$ and $x_2$ of $x$ and an edge between the two copies $y_1$ and $y_2$ of~$y$. This yields a graph $F'$. As $F$ is bipartite and has girth at least $g$, we find that $F'$ is bipartite and has girth at least~$g$ as well. The construction also gives us that $x_1$ and $y_2$ belong to the same bipartition class of $F'$. See also Figure~\ref{thm:girth_MC}.

Now consider an edge $uv$ in $G$. The {\it $F'$-replacement} of $uv$ is obtained from the graphs~$G$ and $F'$ by identifying $u$ with $x_1$ and $v$ with $y_2$. We do an $F'$-replacement on every edge of~$G$. This yields the graph $G'$. 
Since $F'$ is bipartite such that the distance between $x_1$ and $y_2$ is even, we find that $G'$ is bipartite as well.
By construction, $G'$ has girth at least $g$ and maximum degree at most~$4\times (14+1) = 60$.

We claim that $G$ has a matching cut if and only if $G'$ has a matching cut. We prove this below, using  Observation~\ref{o}-(i) implicitly.

First suppose that $G$ has a matching cut, so $G$ has a valid red-blue colouring $c$. For every edge $uv$ in $G$ we do as follows.
Let $F'$ with copies $F_1$ and $F_2$ of $F$ be the corresponding $F'$-replacement applied on $uv$. 
If $u$ and $v$ have the same colour, then we colour every vertex of $V(F')\setminus \{u,v\}$ with that colour. 
If $u$ and $v$ are coloured differently, say $u$ is red and $v$ is blue, then we colour every vertex in $V(F_1)$ red and every vertex in $V(F_2)$ blue. 
This yields a valid red-blue colouring of $G'$. Hence, $G'$ has a matching cut.

Now suppose that $G'$ has a matching cut, so $G'$ has a valid red-blue colouring $c'$. As $F$ is immune, every copy of it in $G'$ is monochromatic. 
Thus, for two vertices $x_1, y_2 \in V(F')$ in some graph $F'$ in $G'$, it holds that $x_1$ and $y_2$ each have a neighbour in $F'$ of the opposite colour if and only if $x_1$ and $y_2$ have different colours in $G'$.
Hence, the restriction of $c'$ to $V(G)$ is a valid red-blue colouring of $G$. Hence, $G$ has a matching cut.
\end{proof}

\noindent
We now focus on the {\sc $d$-Cut} problem for arbitrary $d\geq 1$, and we show how Theorem~\ref{thm:girth_MC} can be generalized in a straightforward way. This requires us to replace Lemma~\ref{l-suitableprimes} by the following lemma.

\begin{lemma}\label{l-suitableprimes_dcut}
There are infinitely many primes $p$, such that
\begin{itemize}
\item [a)] $p\geq 13$,
\item [b)] $p \equiv 1 \bmod 4$, and
\item [c)] there exists an infinite set $Q_p$ such that the following holds for every $q\in Q_p$:  
\begin{itemize}
\item[i.] $q>p$ 
\item[ii.] $q \equiv 1 \bmod 4$, and 
\item[iii.] the Legendre symbol $ \left(\frac{q}{p}\right) = -1$.
\end{itemize}
\end{itemize}
\end{lemma}

\begin{proof}
We first apply Theorem~\ref{thm:dirichlet} to $a = 13$ and $b = 20$, which are coprime. In this way we find that the sequence $13+20k_1$, for $k_1 \in \mathbb{N}$, contains infinitely many primes $p\geq 13$. Hence, condition~a) is satisfied. Since $p=13+20k_1$ for some integer $k_1\geq 0$, we find that $p \equiv 1 \bmod 4$. Hence, condition~b) is satisfied as well.

We now show condition~c). We apply Theorem~\ref{thm:dirichlet} to $a = 5$ and $b = 4p$. In this way we find that the sequence $5+ 4pk_2$, for $k_2 \in \mathbb{N}$, contains infinitely many primes $q$, so in particular an infinite set $Q_p$ with $q>p$ for every $q\in Q_p$. This shows that condition~c)i is satisfied. For every $q\in Q_p$ it holds that $q=5+4pk_2$ for some $k_2\geq 0$ and thus $q \equiv 1 \bmod 4$. Hence, condition~c)ii is satisfied as well.

Finally, we will show condition~c)iii by proving that $\left(\frac{q}{p}\right)= -1$ holds for every $q\in Q_p$. First, we recall that 
if $a \equiv  b \bmod c$, then $\left(\frac{a}{c}\right) = \left(\frac{b}{c}\right)$ (see~\cite{Sh08}).
We will also use the law of reciprocity, of which several equivalent versions exist in the literature. We use the following partial version (which can be found in~\cite{Sh08}). For every two primes $a>2$ and $b>2$, it holds that
\[\left(\frac{a}{b}\right) =  \left(\frac{b}{a}\right)\; \text{if $a \equiv 1 \bmod 4$ or $b \equiv 1 \bmod 4$}.\]
Combining both properties of the Legendre symbol we can now deduce that
\[\begin{array}{lcl}
\displaystyle \left(\frac{q}{p}\right) & = &\displaystyle \left(\frac{5+4pk_2}{13+20k_1}\right)\\[12pt] 
& = &\displaystyle \left(\frac{5+4k_2(13+20k_1)}{13+20k_1}\right)\\[12pt]
&= &\displaystyle\left(\frac{5}{13+20k_1}\right)\\[12pt] 
&= &\displaystyle\left(\frac{13+20k_1}{5}\right) \\[12pt]
&= &\displaystyle\left(\frac{3}{5}\right)\\[12pt] 
&\equiv &-1 \bmod 3.
\end{array}\]
Thus, condition~c)iii is satisfied as well, and we completed the proof of the lemma.
\end{proof}

\noindent
We are now ready to generalize Theorem~\ref{thm:girth_MC}.

\begin{theorem}\label{t-gen}
For every integer $d\geq 1$ and every integer $g \geq 3$, there is a function~$f(d)$ that only depends on $d$, such that {\sc $d$-Cut} is \NP-complete for bipartite graphs of girth at least~$g$ and maximum degree at most $f(d)$.
\end{theorem}

\begin{proof}
Let $d\geq 1$.
By Observation~\ref{o-mi}, every graph $G$ with $h(G)> d$ is $d$-immune. Hence, we can construct a $(p+1)$-regular bipartite gadget~$F$ using the arguments from the proof of Lemma~\ref{lem:girth_matching_immune}, where Lemma~\ref{l-suitableprimes_dcut} plays the role of Lemma~\ref{l-suitableprimes}. That is, we can choose~$p$ such that 

\begin{itemize}
\item [(i)] $p$ is a prime congruent to $1 \bmod 4$, and
\item [(ii)] $\frac{1}{2}(p+1-\lambda_2(G))\geq \frac{1}{2}(p+1-2\sqrt{p})>d$.
\end{itemize}

\noindent
We now reduce from {\sc $d$-Cut}.  Gomes and Sau~\cite{GS21} proved that {\sc $d$-Cut} is \NP-hard for graphs of maximum degree at most $2d+2$. Hence, we may assume that the instance~$G$ from {\sc $d$-Cut} has maximum degree at most $2d+2$. By applying the arguments of the proof of Theorem~\ref{thm:girth_MC}, we can use $F$ and $G$ to construct for every integer~$g\geq 3$, a bipartite graph $G'$ of girth at least $g$ and maximum degree at most $f(d)$, such that $G$ has a $d$-cut if and only if $G'$ has a $d$-cut.
\end{proof}

\noindent
{\bf Remark 2.} Since it is not possible to give the smallest prime $p$ satisfying conditions (i) and (ii) in the proof of Theorem~\ref{t-gen}, we can merely state that the maximum degree of $G'$ is bounded by some function $f(d)$ that only depends on $d$.

\subsection{Disconnected Perfect Matching}\label{s-32}

We now show that {\sc Disconnected Perfect Matching} is \NP-complete for graphs of arbitrarily large fixed girth and bounded maximum degree. Our proof uses some new ideas, but we note that it is also possible to use a similar approach as in the proof of Theorem~\ref{thm:girth_MC}. However, this will lead to a slightly worse bound on the maximum degree, namely $75$ instead of $74$.

We first define some useful auxiliary graphs.
Fix $g \geq 3$ such that $g$ is divisible by $6$ and \nicefrac{$g$}{$6$} is odd.  
Let $F$ be a $14$-regular immune bipartite graph with girth at least $2(g+1)$ containing a perfect matching. 
We note that $F$ exists by Lemma~\ref{lem:girth_matching_immune}, and as $F$ has constant size, $F$ can be found in constant time. 

Let $s$ and $t$ be two designated vertices in $F$ at distance at least $g+1$. We fix a perfect matching $M$ of $F$. Let $x \in N_{F}(s)$ and $y \in N_{F}(t)$ be the (unique) neighbours of $s$ and $t$ in $M$. Then, since $s$ and $t$ are at distance at least $g+1$, $x$ and $y$ are at distance at least $g-1$. We add the edge $xy$ and denote the resulting graph by $F(s,t)$; see also Figure~\ref{fig:girthgadget}a). 

We make the following observation.

\begin{lemma}\label{l-fst}
The graph $F(s,t)$ is immune, bipartite and has girth at least $g$. Moreover, both $F(s,t)$ and $F(s,t)-\{s,t\}$ contain a perfect matching.
\end{lemma}

\begin{proof}
As $F$ is immune, $F(s,t)$ is immune. By our choice of $x$ and $y$, we find that $F(s,t)$ is bipartite and has girth at least~$g$. The fixed perfect matching $M$ of $F$ is of course also a perfect matching of $F(s,t)$.
The new edge $xy$ ensures that $F(s,t)-\{s,t\}$ contains a perfect matching as well; indeed, we can take $(M\setminus \{sx,ty\})\cup \{xy\}$. 
\end{proof}

\noindent
We now take $k = \frac{g}{6}$ copies $F(s_1,t_1), \dots, F(s_k,t_k)$ of $F(s,t)$ and identify $s_{i+1}$ with $t_{i}$ for all $i \in \{1, \dots, k-1\}$. We set $s=s_1$ and $t=t_k$ and call the resulting graph $H(s,t)$; see also Figure~\ref{fig:girthgadget}a).

In the following lemma we show some useful properties of $H(s,t)$ and $H(s,t)-\{s,t\}$.

\begin{figure}[t]
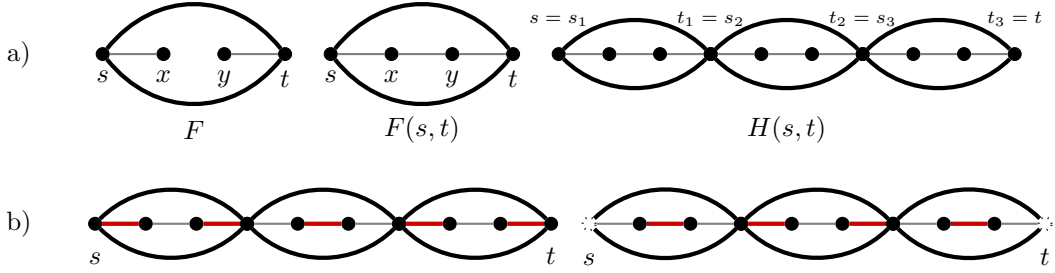

\centering
\include{gadgetgirth}
\caption{a) Illustration of the graphs $F$ (left), $F(s,t)$ (middle) and $H(s,t)$ (right). b) Illustration of how the edges in the perfect matching of $H(s,t)$ (left) resp. $H(s,t)-\{s,t\}$ (right) are chosen (these edges are represented as red, thick edges).\label{fig:girthgadget}}
\end{figure}  

\begin{lemma} \label{lem:girthgadget}
The graph $H(s,t)$ is immune, bipartite, and has girth at least~$g$. Moreover, $\dist(s,t) \geq \frac{g}{2}$, and both $H(s,t)$ and $H(s,t)-\{s,t\}$ contain a perfect matching.
\end{lemma}

\begin{proof}
By taking two immune graphs and identifying two of their vertices (one from each graph), we obtain another immune graph. The graph $H(s,t)$ is obtained by repeatedly applying this operation on copies of $F(s,t)$, which is immune by Lemma~\ref{l-fst}. Hence, we conclude that $H(s,t)$ is immune.

By Lemma~\ref{l-fst}, every graph $F(s_i,t_i)$ for $i \in \{1,\dots, k\}$ is bipartite and has girth at least~$g$. By constructing $H(s,t)$ from these graphs $F(s_i,t_i)$, we do not create any new cycles. Thus, $H(s,t)$ also is bipartite and has girth at least $g$. 
In every graph $F(s_i,t_i)$ we have $\dist(s_i, t_i) = 3$, thus $\dist(s,t) = 3\times \frac{g}{6} = \frac{g}{2}$.

To see that $H(s,t)$ and $H(s,t) - \{s,t\}$ contain both a perfect matching, we will use the fact that both $F(s_i,t_i)$ and $F(s_i,t_i) -\{s_i, t_i\}$ contain a perfect matching by Lemma~\ref{l-fst}. For $H(s,t)$, we take a perfect matching of $F(s_1,t_1)$, $F(s_2,t_2) - \{s_2,t_2\}$, $F(s_3,t_3)$ and continue this alternation until $F(s_k,t_k)$, see Figure~\ref{fig:girthgadget}b (left). For $H(s,t) - \{s,t\}$ we alternate as well but we start and end with $F(s_1,t_1) - \{s_1,t_1\}$ and $F(s_k,t_k) - \{s_k,t_k\}$, see Figure~\ref{fig:girthgadget}b (right).
\end{proof}

\noindent
We are now ready to prove the following result.

\begin{theorem}\label{thm:girth_DPM}
For every integer $g \geq 3$, {\sc Disconnected Perfect Matching} is \NP-complete for bipartite graphs of girth at least~$g$ and maximum degree at most~$74$. 
\end{theorem}

\begin{proof}
Let $g\geq 3$. 
We reduce from {\sc Matching Cut} for bipartite graphs of girth at least $g$ and maximum degree at most~$60$, which is \NP-complete by Theorem~\ref{thm:girth_MC}. 
Similar to Theorem~\ref{thm:girth_MC}, we may assume without loss of generality that $g$ is divisible by $12$. Let $G$ be a bipartite graph of girth at least $g$ and maximum degree~$60$. We construct a graph $G'$ by taking two copies $G_1$ and $G_2$ of $G$, where we connect every vertex $v \in V(G_1)$ and its copy $v' \in V(G_2)$ using the graph $H(v,v')$.

To see that $G'$ has girth at least $g$, we consider first $G_1$ and $G_2$, which both have girth at least $g$. Any cycle containing vertices from both copies has to pass twice through a graph $H(s,t)$. Thus, it will always have length at least $g$, and so $G'$ has girth $g$.
Every vertex inside $H(s,t)$ has degree at most $28$, whereas $s$ and $t$ only have degree $14$.
The degree of a vertex $v \in V(G_1) \cup V(G_2)$ is the degree of the vertex in the original graph $G$ plus the degree in the graph $H(v,v')$.  
Thus, the degree of $v$ is at most~$74$. 

We also claim that $G'$ is bipartite. For a contradiction, assume that $G'$ has an odd cycle~$C$. As $G_1$ and $G_2$ are both bipartite,  $C'$ must contain vertices from $G_1$ and $G_2$. 
Note that $C$ passes through an even number of graphs $H(v,v')$, where $v \in V(G_1)$ and $v' \in V(G_2)$.
Hence, the number of edges in $E(G)\setminus(E(G_1) \cup E(G_2))$ is even. 
Since the graph $H(v,v')$ always connects a vertex $v$ in $G_1$ and its copy $v'$ in $G_2$ we can find an odd cycle $C'$ in $G_1$, consisting of the edges in $C \cap E(G_1)$ and the edges in $E(G_1)$ corresponding to the edges from $C \cap E(G_2)$, a contradiction.

Finally, we show that $G$ admits a matching cut if and only if $G'$ admits a disconnected perfect matching. Consider some vertex $v \in V(G_1)$ and its copy $v' \in V(G_2)$. Since $v$ and $v'$ are connected by the graph $H(v,v')$, which is immune, $v$ and $v'$ will always have the same colour in any valid red-blue colouring of $G'$. Thus, $G_1$ and $G_2$ will be coloured the same in any valid red-blue colouring. Now, if $G$ admits a perfect-extendable red-blue colouring, then~$G$ admits a valid red-blue colouring, as it suffices to colour $G$ the same as $G_1$ (or $G_2$). 

Conversely, if $G$ admits a valid red-blue colouring, then we obtain a perfect-extendable colouring of $G'$ as follows. We colour $G_1$ and $G_2$ the same as $G$. Notice that we colour the immune graphs connecting two copies of the same vertex such that they are monochromatic. This gives us a valid red-blue colouring of $G'$, i.e.\ a matching cut $M$ in $G'$. It remains to show that the matching cut is contained in a perfect matching of $G'$. Since the colourings of $G_1$ and $G_2$ are the same, we have that whenever a blue vertex $v \in V(G_1)$ is matched with a red vertex $u\in V(G_1)$, i.e.\ $vu\in M$, then their copies $v'\in V(G_2)$ and $u'\in V(G_2)$ are matched as well, i.e.\ $v'u'\in M$. By Lemma~\ref{lem:girthgadget}, we know that $H(v,v')-\{v,v'\}$ and $H(u,u')-\{u,u'\}$ contain both a perfect matching which we may add to $M$. For every $v \in V(G_1)$ with no neighbour of the other colour, we know that its copy $v' \in V(G_2)$ has no neighbour of the other colour either. Thus, we can use that $H(v,v')$ contains a perfect matching by Lemma~\ref{lem:girthgadget} and add it to $M$. Repeatedly doing this yields a perfect matching of $G'$ containing~$M$.
\end{proof} 

\subsection{Perfect Matching Cut}\label{s-33}

In any perfect red-blue colouring of a connected graph $G=(V,E)$, a vertex $v\in V$ of degree~$2$ has exactly one neighbour coloured the same as itself and exactly one neighbour coloured differently than itself. One can use this observation to prove the following lemma. This lemma was implicit in~\cite{LT21}. Its proof is short, and we include it for completeness.
 
\begin{lemma}[\cite{LT21}]\label{lem:subdiv}
Let $G=(V',E')$ be the graph obtained from a connected graph~$G$ by $4$-subdividing an edge $e$ of $G$.  
Now, $G$ has a perfect matching cut if and only if $G'$ has a perfect matching cut.
\end{lemma}

\begin{proof}
Let $e=uv$ and denote the resulting path in $G'$ by $uu_1u_2u_3u_4v$.
First suppose $G$ has a perfect matching cut, so $G$ has a perfect red-blue colouring~$c$. If $u$ and $v$ are coloured alike, say $u$ and $v$ are red,
then we colour $u_1$ red, $u_2$ blue, $u_3$ blue and $u_4$ red. Else we may assume that $u$ is red and $v$ is blue. In that case we colour $u_1$ blue, $u_2$ blue, $u_3$ red and $u_4$ red. In both cases we obtain a perfect red-blue colouring~$c'$ of $G'$. So $G'$ has a perfect matching cut.

Now suppose that $G'$ has a perfect matching cut, so $G'$ has a perfect red-blue colouring~$c'$. 
If $u$ and $v$ are coloured alike, say $u$ and $v$ are red, then $u_1$ must be red, $u_2$ blue, $u_3$ blue and $u_4$ red. Hence, the restriction of $c'$ to $V$ is a perfect red-blue colouring of $G$. Else we may assume that $u $ is red and $v$ is blue. In that case $u_1$ must be blue, $u_2$ blue, $u_3$ red and $u_4$ red. Again, the restriction of $c'$ to $V$ is a perfect red-blue colouring of $G$. Hence, in both cases we find that $G$ has a perfect matching cut.
\end{proof}

\noindent
A simple proof for showing that {\sc Perfect Matching Cut} is \NP-complete for graphs of girth at least~$g$ is to apply Lemma~\ref{lem:subdiv}, say, $g$ times on each edge of the input graph. Recall that the more involved gadget of Le and Telle~\cite{LT21} yields
\NP-completeness even for subcubic bipartite graphs of girth at least $g$.

\section{Hardness for Forbidden Subdivided H-Graphs}\label{s-h}

We show that {\sc Matching Cut} (Section~\ref{s-41}), {\sc Disconnected Perfect Matching} (Section~\ref{s-42}) and {\sc Perfect Matching Cut} (Section~\ref{s-43}) are \NP-complete for $(H_1^*,\ldots,H_i^*)$-free graphs, for every $i\geq 1$. 

\subsection{Matching Cut}\label{s-41}

Let $uv$ be an edge of a graph~$G$. We define an edge operation as displayed in Figure~\ref{fig:gadget2}, which when applied on $uv$ will replace $uv$ in~$G$ by the subgraph $T_{uv}^i$. Note that in the new graph, the only vertices from $T_{uv}^i$ that may have neighbours outside $T_{uv}^i$ are $u$ and $v$.

\begin{figure}[ht]
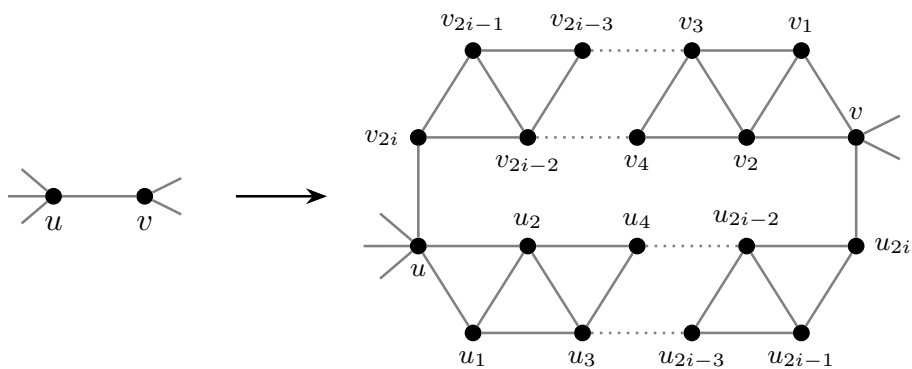

\centering
\include{gadget2}
\caption{The edge $uv$ (left) which we replace by the subgraph $T_{uv}^i$ (right).}\label{fig:gadget2}
\end{figure}

\begin{theorem}\label{thm:subdividedH_MC}
For all $i\geq 1$, {\sc Matching Cut} is \NP-complete for $(H_1^*,\ldots,H_i^*)$-free graphs.
\end{theorem}

\begin{proof}
Fix  $i\geq 1$. Reduce from {\sc Matching Cut}. Let $G=(V,E)$ be a connected graph. Replace every $uv\in E$ by the graph $T_{uv}^i$ (see Figure~\ref{fig:gadget2}). This yields the graph $G'=(V',E')$.

We claim that $G'$ is $(H_{1}^*,\ldots,H_i^*)$-free. For a contradiction, assume that $G'$ contains an induced $H_{i'}^*$ for some $1\leq i'\leq i$.
Then $G'$ contains two vertices $x$ and $y$ that are centers of an induced claw, as well as an induced path from $x$ to $y$ of length $i'$. All vertices in $V'\setminus V$ are not centers of any induced claw. Hence, $x$ and $y$ belong to $V$. By construction, any shortest path between two vertices of $V$ has length at least $i+1$ in $G'$, a contradiction.

We claim that $G'$ has a matching cut if and only if $G$ has a matching cut. First suppose $G'$ has a matching cut $M'$, so $G'$ has a valid red-blue colouring $c'$. We prove a claim for $G'$:

\begin{claim1}
\label{claim-gadget2}
For every edge $uv\in E(G)$ it holds that
\begin{itemize}
\item[(a)] either $c'(u)=c'(v)$, and then $T_{uv}^i$ is monochromatic, or
\item[(b)] $c'(u)\neq c'(v)$, and then $c'$ colours $u_1,\ldots,u_{2i}$ with the same colour as $u$, while $c'$ colours all vertices of $T_{uv}^i-\{u,u_1,\ldots,u_{2i}\}$ with the same colour as $v$, and moreover, $uv_{2i},vu_{2i}\in M'$.
\end{itemize}
\end{claim1}

\noindent
\begin{claimproof}
First assume $c'(u)=c'(v)$, say $c'$ colours $u$ and $v$ red. As any clique of size at least~$3$ is monochromatic, all vertices in $T_{uv}^i$ are coloured red, so $T_{uv}^i$ is monochromatic.

Now assume  $c'(u)\neq c'(v)$, say $u$ is red and $v$ blue. As before, we find that all vertices $u_1,\ldots,u_{2i}$ have the same colour as $u$, so are red, while all vertices $v_1,\ldots,v_{2i}$ have the same colour as $v$, so are blue.
By definition, every edge $xy$ with $c'(x)\neq c'(y)$ belongs to $M'$, so $uv_{2i},vu_{2i}\in M'$.
\end{claimproof}

\noindent
We construct a subset $M\subseteq E$ in $G$ as follows. We add an edge $uv\in E$ to $M$ if and only if $c'(u)\neq c'(v)$ in $G'$.
We now show that $M$ is a matching in $G$. Let $u\in V$. For a contradiction, suppose that $M$ contains edges $uv$ and $uw$ for $v\neq w$. Then $c'(u)\neq c'(v)$ and $c'(u)\neq c'(w)$. By Claim~\ref{claim-gadget2}, we find that $M'$ matches $u$ in $G'$ to vertices in $T_{uv}^i$ and $T_{uw}^i$, contradicting our assumption that $M'$ is a matching (cut). Hence, $M$ is a matching.

Now let $c$ be the restriction of $c'$ to $V$. If $c$ colours every vertex of $G$ with one colour, say red, then $c'$ would also colour every vertex of $G'$ red by Claim~\ref{claim-gadget2}, contradicting the validity of $c'$. Hence, $c$ uses both colours.
Moreover, for every $uv\in E$, the following holds: if $c(u)\neq c(v)$, then $c'(u)\neq c'(v)$ and thus $uv\in M$. Hence, as $M$ is a matching, $c$ is valid, and thus $M$ is a matching cut of $G$.

Conversely, assume that $G$ admits a matching cut, so $V$ has a valid red-blue colouring $c$. 
We construct a red-blue colouring $c'$ of $V'$ as follows.

\begin{itemize}
\item For every edge $uv\in E$ with $c(u)=c(v)$, we let $c'(x)=c(u)$ for every $x\in V(T_{uv}^i)$. 
\item For every edge $uv\in E$ with $c(u)\neq c(v)$, we let $c'(u)=c'(u_1)=\cdots = c'(u_{2i})=c(u)$ and $c'(v)=c'(v_1)=\cdots =c'(v_{2i})=c(v)$.\\[-10pt]
\end{itemize}

\noindent
As $c$ is valid, $c$ uses both colours and thus by construction, $c'$ uses both colours. Let $u\in V$. Again as $c$ is valid,
$c(u)\neq c(v)$ holds for at most one neighbour $v$ of $u$ in $G$. Hence, by construction, $u$ belongs to at most one non-monochromatic gadget $T_{uv}^i$. Thus, $c'$ colours in $G'$ at most one neighbour of $u$ with a different colour than $u$.
Let $u\in V'\setminus V$. By construction, we find again that $c'$ colours at most one neighbour of $u$ with a different colour than $u$. 
Hence, $c'$ is valid, and so $G'$ has a matching cut.
This completes the proof of Theorem~\ref{thm:subdividedH_MC}.
\end{proof}

\subsection{Disconnected Perfect Matching}\label{s-42}

\begin{figure}[ht]
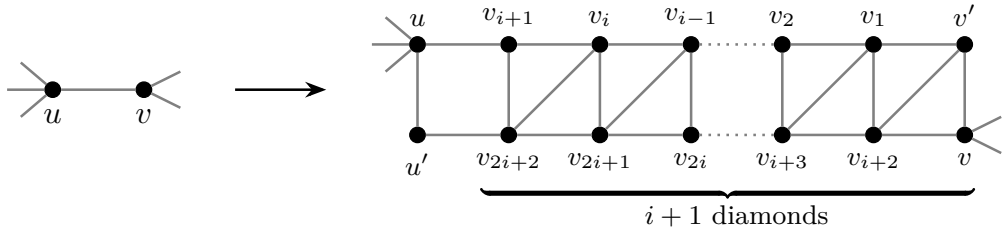

\centering
\include{gadget}
\caption{The edge $uv$ (left) which we replace by the subgraph $G_{uv}^i$ (right).}\label{fig:gadget}
\end{figure}

\noindent
We show our next result in a similar way as Theorem~\ref{thm:subdividedH_MC}. First we define the edge operation displayed in Figure~\ref{fig:gadget}. This operation replaces an edge $uv$ in a graph $G$ by the subgraph $G_{uv}^i$ for some integer $i\geq 1$. Note that in the resulting graph, the only vertices from $G_{uv}^i$ that may have neighbours outside $G_{uv}^i$ are $u$ and $v$.

\begin{theorem}\label{thm:subdividedH_DPM}
For every $i\geq 1$, {\sc Disconnected Perfect Matching} is \NP-complete for $(H_1^*,\ldots,H_i^*)$-free graphs.
\end{theorem}

\begin{proof}
Fix an integer $i\geq 1$. As the class of $(H_1^*,\ldots,H_i^*)$-free graphs is contained in the class of $(H_1^*,\ldots,H_{i-1}^*)$-free graphs if $i\geq 2$, we may assume without loss of generality that $i$ is even (we need this assumption at a later place in our proof). 
 We reduce from {\sc Disconnected Perfect Matching} itself. Let $G=(V,E)$ be a connected graph. 
 
Let $uv$ be an edge of $G$. We recall the edge operation displayed in Figure~\ref{fig:gadget}, which when applied on $uv$ replaces $uv$ by the subgraph $G_{uv}^i$. Recall that in the resulting graph, $u$ and $v$ are the only vertices from $G_{uv}^i$ that may have neighbours outside $G_{uv}^i$.

In $G$ we now replace every edge $uv\in E$ by the graph $G_{uv}^i$. Let $G'=(V',E')$ be the resulting graph. 

We claim that $G'$ is $(H_{1}^*,\ldots,H_i^*)$-free. For a contradiction, assume that $G'$ contains an induced $H_{i'}^*$ for some $1\leq i'\leq i$. Then $G'$ contains two vertices $x$ and $y$ that are centers of an induced claw, as well as an induced path from $x$ to $y$ of length $i'$. The only vertices in $V'\setminus V$ that are the center of an induced claw are the vertices $v_{2i+2}$ (see Figure~\ref{fig:gadget}). Suppose for a contradiction that $x = v_{2i+2}$. Since the graph $G[\{u,u',v_{2i+2}, v_{2i+1}, v_{i+1}\}]$ contains an induced $C_4$, the induced path from $x$ to $y$ may not contain $u$. Thus, either $y = v$ or the path from $x$ to $y$ passes through $v$. Since a shortest path from $v_{2i+2}$ to $v$ has length $i+1$, we get a contradiction.
Hence, $x$ and $y$ belong to $V$. By construction, any shortest path between two vertices of $V$ has length at least $i+3$ in $G'$, a contradiction.

We claim that $G'$ has a disconnected perfect matching if and only if $G$ has a disconnected perfect matching. 
First assume that $G'$ has a disconnected perfect matching $M'$, so $V'$ has a perfect-extendable red-blue colouring $c'$. We prove the following claim for $G'$:

\medskip
\noindent \textcolor{lipicsGray}{\ensuremath{\blacktriangleright}} {\bf \sffamily\bfseries Claim~\ref{thm:subdividedH_DPM}.1.}
For every edge $uv\in E(G)$, it holds that
\begin{itemize}
\item[(a)] either $c'(u)=c'(v)$, and then $G_{uv}^i$ is monochromatic, or
\item[(b)] $c'(u)\neq c'(v)$, and then $c'$ colours $u'$  and all neighbours of $u$ in $G-v$ with the same colour as $u$, while $c'$ colours all vertices of $G_{uv}^i-\{u',u\}$ with the same colour as $v$, and moreover,  $uv_{i+1}\in M'$ and either $vv'\in M'$ or $vv_{i+2}\in M'$.
\end{itemize}

\begin{claimproof}
First assume $c'(u)=c'(v)$, say $c'$ colours $u$ and $v$ red. As any clique of size at least~$3$ is monochromatic, all vertices $v',v_1,\ldots,v_{2i+2}$ are coloured the same as $v$, so they are red. Now $u'$ has two red neighbours, and thus $u'$ must be red as well. Hence, $G_{uv}^i$ is monochromatic.

Now assume  $c'(u)\neq c'(v)$, say $u$ is red and $v$ blue. As before, we find that $v',v_1,\ldots,v_{2i+2}$ are all coloured the same as $v$, so they are blue. Now $u$, which is red, has one blue neighbour, namely $v_{i+1}$, and thus, $u'$ must be red. So $u$ and $u'$ have the same colour.
By definition, every non-monochromatic edge belongs to~$M'$. Hence, $uv_{i+1}\in M'$. Recall that $M'$ is perfect and observe that $|V(G_{uv}^i)|$ is even. As the vertices of $G_{uv}^i-\{u,v\}$ have no neighbours outside $G_{uv}^i-\{u,v\}$, these two facts imply that either $vv'\in M'$ or $vv_{i+2}\in M'$. If  $c'(u)\neq c'(w)$ for some neighbour $w$ of $u$ in $G-v$ then, as we just argued, $M'$ matches~$u$ with a vertex of $G_{uw}^i$, a contradiction, as $M'$ is a matching in $G'$. This proves Claim~\ref{thm:subdividedH_DPM}.1. 
\end{claimproof}

\noindent
We now construct a subset $M\subseteq E$ in $G$ by the following rule: add an edge $uv\in E$ to $M$ if and only if in $G'$, either $c'(u)=c'(v)$ and $M'$ matches both $u$ and $v$ with vertices from $G_{uv}^i$, or $c'(u)\neq c'(v)$.

We now show that $M$ is a perfect matching in $G$. Let $u\in V$. First suppose $u$ belongs in $G'$ to a non-monochromatic graph $G_{uv}^i$. By Claim~\ref{thm:subdividedH_DPM}.1, we have $c'(u)\neq c'(v)$ and $c'(u)=c'(w)$ for every neighbour $w$ of $u$ in $G-v$, and moreover, $M'$ matches both $u$ and $v$ with vertices from $G_{uv}^i$. Hence, $M$ matches $u$ to $v$. As  $c'(u)=c'(w)$ for every neighbour $w$ of $u$ in $G-v$, each $G_{uw}^i$ is monochromatic. However, as $M'$ already matched $u$ to some vertex from $G_{uv}^i$, we find that $M'$ (being a matching) does not match $u$ to any vertices of any $G_{uw}^i$ with $w\neq v$. Hence, our rule does not add any edges $uw$ with $w\neq v$ to $M$. So $M$ matches $u$ only to $v$. 

Now suppose that $u$ only belongs to monochromatic graphs $G_{uv}^i$. By Claim~\ref{thm:subdividedH_DPM}.1, we have $c'(u)= c'(v)$ for every neighbour $v$ of $u$ in $G$.
As $M'$ is a perfect matching, $u$ has exactly one neighbour $v$ in $G$, such that $M'$ matches $u$ with a vertex of $G_{uv}^i$. As $|V(G_{uv}^i)|$ is even and the vertices of $G_{uv}^i-\{u,v\}$ have no neighbours outside $G_{uv}^i-\{u,v\}$, this means that $M'$ also matches $v$ with a vertex from $G_{uv}^i$. Hence, $M$ matches $u$ to $v$ and to no other vertex of $G$. We conclude that $M$ is indeed a perfect matching in $G$.

\begin{figure}[t]
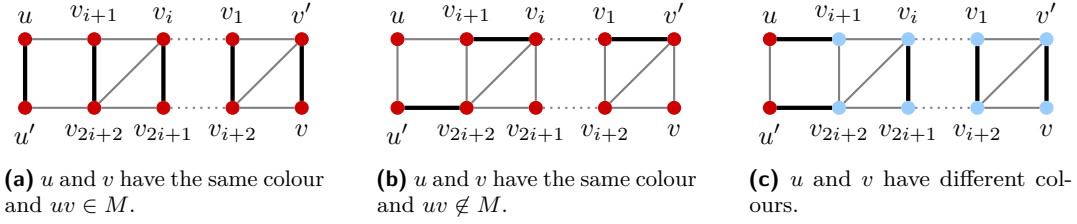

\centering
\include{gadgetcoloured}
\vspace*{0.5cm}
\caption{The different colourings and matchings in the graph $G_{uv}^i$.}\label{fig-gadgetcoloured}
\end{figure}

Finally, let $c$ be the restriction of $c'$ to $V$. If $c$ colours every vertex of $G$ with one colour, say red, then $c'$ would also colour every vertex of $G'$ red by Claim~\ref{thm:subdividedH_DPM}.1, contradicting the perfect-extendability (and thus validity) of $c'$. Hence, $c$ uses both colours.
For every $uv\in E$, the following holds: if $c(u)\neq c(v)$, then $c'(u)\neq c'(v)$ and thus $uv\in M$. Hence, as $M$ is a (perfect) matching, $c$ is perfect-extendable. Thus $M$ is a disconnected perfect matching of $G$.

\medskip
\noindent
Now, assume that $G$ has a disconnected perfect matching $M$, so $V$ has a perfect-extendable red-blue colouring $c$. 
We construct a red-blue colouring $c'$ of $V'$ and a matching~$M'$ in $G'$ as follows:

\begin{itemize}
\item For every edge $uv\in E$ with $c(u)=c(v)$, we let $c'(x)=c(u)$ for every $x\in V(G_{uv}^i)$. 
\begin{itemize}
\item If $uv\in M$, then we add $vv',v_1v_{i+2},\ldots,v_{i+1}v_{2i+2},uu'$ to $M'$; see also Figure~\ref{fig-gadgetcoloured}(a).
\item If $uv\not\in M$, then we add $v'v_1, v_2v_3, \dots, v_{2i}v_{2i+1},v_{2i+2}u'$ to $M'$ (recall that $i$ is even, so this is possible); see also Figure~\ref{fig-gadgetcoloured}(b).
\end{itemize}
\item For every edge $uv\in E$ with $c(u)\neq c(v)$, we let $c'(u)=c'(u')=c(u)$ and $c'(x)=c(v)$ for every $x\in V(G_{uv}^i)\setminus \{u,u'\}$. We add $uv_{i+1},u'v_{2i+2}$ to $M'$ as well as $vv',v_1v_{i+2},\ldots, v_{i}v_{2i+1}$; see also Figure~\ref{fig-gadgetcoloured}(c).
\end{itemize}

\noindent
As $c$ is valid, $c$ uses both colours and thus by construction, $c'$ uses both colours. Let $u\in V$. Again as $c$ is valid,
$c(u)\neq c(v)$ holds for at most one neighbour $v$ of $u$ in $G$. Hence, by construction, $u$ belongs to at most one non-monochromatic gadget $G_{uv}^i$. Thus, $c'$ colours in $G'$ at most one neighbour of $u$ with a different colour than $u$.
Let $u\in V'\setminus V$. By construction, we find again that $c'$ colours at most one neighbour of $u$ with a different colour than $u$. 
Hence, $c'$ is valid. Again by construction, $M'$ is a perfect matching containing all edges $xy$ of $G'$ with $c'(x)\neq c'(y)$, and thus, $M'$ is a disconnected perfect matching of $G'$. 
This completes the proof of Theorem~\ref{thm:subdividedH_DPM}.
\end{proof}
 
\subsection{Perfect Matching Cut}\label{s-43}

We apply Lemma~\ref{lem:subdiv} (implicit in~\cite{LT21}) sufficiently times on every edge of a subcubic bipartite graph of girth at least~$g$ and combine this with the \NP-completeness of {\sc Perfect Matching Cut} for the class of subcubic bipartite graph of girth at least~$g$~\cite{LT21}. Note that applying Lemma~\ref{lem:subdiv} does not change the degree of a graph. Hence, Le and Telle essentially proved the following:

\begin{theorem}[\cite{LT21}]\label{thm:subdividedH_PMC}
For every $i\geq 1$ and $g\geq 3$, {\sc Perfect Matching Cut} is \NP-complete for $(H_1^*,\ldots,H_i^*)$-free subcubic bipartite graphs of girth at least~$g$.
\end{theorem}

\section{Consequences and Open Problems}\label{s-con}

We give some consequences of our new results on {\sc Matching Cut}, {\sc Disconnected Perfect Matching} and {\sc Perfect Matching Cut} for $H$-free graphs and $H$-subgraph-free graphs.

\subsection{H-Free Graphs}

We give three up-to-date classifications for $H$-free graphs by combining the results from~\cite{Bo09,BP,Ch84,LL23,LR03,LPR22b,LPR22a,LT21,Mo89} (see Section~\ref{s-known}) with our new results. That is, we took the three state-of-the-art theorems in~\cite{LPR22b} and added both the result for $H_i^*$-free graphs and the result for $(3P_6,2P_7,P_{14})$-free graphs from~\cite{LL23}. For {\sc Disconnected Perfect Matching} we also added the new result for $C_s$-free graphs for even~$s$, as implied by our girth result. 
We also compare the three partial classification with a recent, full classification of the optimization problem {\sc Maximum Matching Cut}~\cite{LPR23}.
We write $G'\ssi G$ if $G'$ is an induced subgraph of~$G$ and $G'\si G$ if $G$ is an induced subgraph of $G'$.

\begin{theorem}\label{t-main1}
For a graph~$H$, {\sc Matching Cut} on $H$-free graphs is 
\begin{itemize}
\item polynomial-time solvable if $H\ssi sP_3+K_{1,3}$ or $sP_3+P_6$ for some $s\geq 0$, and
\item \NP-complete if $H\si K_{1,4}$, $P_{14}$, $3P_5$, $2P_7$, $C_r$ for some $r\geq 3$ or $H_j^*$ for some $j\geq 1$.
\end{itemize}
\end{theorem}

\begin{theorem}\label{t-main2}
For a graph~$H$, {\sc Disconnected Perfect Matching} on $H$-free graphs is 
\begin{itemize}
\item polynomial-time solvable if $H\ssi K_{1,3}$ or $P_5$, and
\item \NP-complete if $H\si K_{1,4}$, $P_{14}$, $3P_6$, $2P_7$, $C_r$ for some $r\geq 3$ or $H_j^*$ for some $j\geq 1$.
\end{itemize}
\end{theorem}

\begin{theorem}\label{t-main3}
For a graph~$H$, {\sc Perfect Matching Cut} on $H$-free graphs is 
\begin{itemize}
\item polynomial-time solvable if $H\ssi sP_4+S_{1,2,2}$ or $sP_4+P_6$ for some $s\geq 0$, and
\item \NP-complete if $H\si K_{1,4}$, $P_{14}$, $3P_6$, $2P_7$, $C_r$ for some $r\geq 3$ or $H_j^*$ for some $j\geq 1$.
\end{itemize}
\end{theorem}

\begin{theorem}\label{t-dichoH}
For a graph~$H$, {\sc Maximum Matching Cut} on $H$-free graphs is 
\begin{itemize}
\item polynomial-time solvable if $H\ssi sP_2+P_6$ for some $s\geq 0$, and
\item \NP-hard if $H\si K_{1,3}$, $2P_3$, $C_r$ for some $r\geq 3$.
\end{itemize}
\end{theorem}

\noindent
A {\it subdivided claw} is a graph obtained from the claw $K_{1,3}$ by subdividing each of its edges zero or more times. 
Let ${\cal S}$ be the class of graphs, each connected component of which is either a path or a {\it subdivided claw}.
From Theorem~\ref{t-main1}, it follows that {\sc Matching Cut} is \NP-complete for $H$-free graphs if $H$ has a cycle, a vertex of degree at least~$4$, or a connected component with two vertices of degree~$3$. Hence,
the remaining open cases for {\sc Matching Cut} on $H$-free graphs are all restricted to cases where $H$ is a graph from ${\cal S}$. The same remark holds for {\sc Disconnected Perfect Matching} due to Theorem~\ref{t-main2}, and for {\sc Perfect Matching Cut} due to Theorem~\ref{t-main3}.

\subsection{H-subgraph-free Graphs}

For a graph~$H$, a graph $G$ is {\it $H$-subgraph-free} if $G$ does not contain $H$ as a subgraph.
Every $H$-subgraph-free graph is $H$-free, 
whereas the reverse direction only holds if $H$ is a complete graph.
For a set ${\cal H}$ of graphs, a graph $G$ is  {\it ${\cal H}$-subgraph-free} if $G$ is $H$-subgraph-free for every $H\in {\cal H}$. 

For an integer $p$, a \textit{$p$-subdivision} of an edge $uv$ in a graph replaces $uv$ by a path from $u$ to~$v$ of length $p+1$.
The {\it $p$-subdivision} of a graph~$G$ is obtained from~$G$ by $p$-subdividing each edge of $G$. 
For a graph class ${\cal G}$, we let ${\cal G}^p$ consist of all the $p$-subdivisions of the graphs in ${\cal G}$.
A graph problem~$\Pi$ is \NP-complete  {\it under edge subdivision of subcubic graphs} if there is an integer~$q\geq 1$ such that the following holds:
if $\Pi$ is \NP-complete for the class~${\cal G}$ of subcubic graphs, then $\Pi$ is \NP-complete for ${\cal G}^{pq}$ for every $p\geq 1$.
Now, $\Pi$ is a {\it C123-problem} if  (C1) $\Pi$ is polynomial-time solvable for graphs of bounded treewidth; (C2) $\Pi$ is \NP-complete for subcubic graphs; and (C3) $\Pi$ is \NP-complete under edge subdivision of subcubic graphs. 

In~\cite{JMOPPSV}, it was shown that for every finite set of graphs ${\cal H}$, any C123-problem $\Pi$ on ${\cal H}$-subgraph-free graphs is polynomial-time solvable if ${\cal H}$ contains a graph from ${\cal S}$ and \NP-complete otherwise.
Le and Telle~\cite{LT21} observed that {\sc Perfect Matching Cut} satisfies C1 and proved C2 and C3 
(see Lemma~\ref{lem:subdiv}). 
Applying the above meta-theorem from~\cite{JMOPPSV} yields:

\begin{theorem}\label{t-dicho}
For any finite set of graphs ${\cal H}$, {\sc Perfect Matching Cut} on ${\cal H}$-subgraph-free graphs is polynomial-time solvable if ${\cal H}$ contains a graph from ${\cal S}$ and \NP-complete otherwise.
\end{theorem}

\subsection{Open Problems}

Apart from completing the classifications of Theorems~\ref{t-main1}--\ref{t-main3}, we also pose the following open problem.

\begin{open}
Classify the computational complexity of {\sc Matching Cut} and {\sc Disconnected Perfect Matching} for ${\cal H}$-subgraph-free graphs.
\end{open}

\noindent
We note that classifications of {\sc Matching Cut} and {\sc Disconnected Perfect Matching} are unknown even for $H$-subgraph-free graphs (so when we forbid only a single graph $H$ as a subgraph). So far, only a partial classification for {\sc Matching Cut} restricted to ${\cal H}$-subgraph-free graphs has been shown~\cite{JMPPSV}. In particular, the transformations for {\sc Matching Cut} and {\sc Disconnected Perfect Matching} from Section~\ref{s-h} do not decrease the girth and yield graphs with many cycles of varying length as subgraphs. Hence, new techniques are needed.

Finally, with our current technique (Lemma~\ref{lem:girth_matching_immune}) we cannot obtain a better bound on the maximum degree of the graphs of arbitrarily large fixed girth in the proofs of Theorems~\ref{thm:girth_MC} and~\ref{thm:girth_DPM}. 

\begin{open}
Can the two maximum degree bounds in Theorems~\ref{thm:girth_MC} and~\ref{thm:girth_DPM} be improved?
\end{open}

\bibliography{ref}
\end{document}

%% file: Hstar.tex
\begin{tikzpicture}
\begin{scope}[scale=1.4]
	\node[vertex] (v1) at (0,2){};
	\node[vertex, label=left:$u$] (v2) at (0,1){};
	\node[vertex] (v3) at (0,0){};
	\node[vertex] (v4) at (1,2){};
	\node[vertex, label=right:$v$] (v5) at (1,1){};
	\node[vertex] (v6) at (1,0){};
	
	\draw[edge](v1)--(v2);
	\draw[edge](v2)--(v5);
	\draw[edge](v2)--(v3);
	\draw[edge](v4)--(v5);
	\draw[edge](v5)--(v6);
	
\end{scope}

\begin{scope}[shift = {(3,0)},scale=1.4]
	\node[vertex] (v1) at (0,2){};
	\node[vertex, label=left:$u$] (v2) at (0,1){};
	\node[vertex] (v3) at (0,0){};
	\node[vertex] (v4) at (3,2){};
	\node[vertex, label=right:$v$] (v5) at (3,1){};
	\node[vertex] (v6) at (3,0){};
	
	\node[vertex] (u1) at (1,1){};
	\node[vertex] (u2) at (2,1){};
	
	\draw[edge](v1)--(v2);
	%\draw[edge](v2)--(v5);
	\draw[edge](v2)--(v3);
	\draw[edge](v4)--(v5);
	\draw[edge](v5)--(v6);
	\draw[edge](v2)--(u1);
	\draw[edge](u2)--(v5);
	\draw[edge](u1)--(1.3,1);
	\node[color = gray](dots) at (1.65,1){$\dots$};
	
	\draw[br](2.925,0.875)--(0.075,0.875);
	%\scriptsize
	\node[](i) at (1.5, 0.6){length $i$};
	
\end{scope}
\end{tikzpicture}

%% file: dpm_girth.tex
\begin{tikzpicture}
\begin{scope}[xscale=1.2]
	\node[vertex, label=above:\strut $x$](x) at (0,0){};
	\node[vertex, label=above:\strut $y$](y) at (2,0){};
	
	\node[](text) at (-0.75,0){$F$};
	
	\draw[ultra thick, black](x) to [bend left = 30] (y);
	\draw[ultra thick, black](x) to [bend right = 30] (y);
	
\end{scope}

\begin{scope}[shift = {(6.75,0.5)}, xscale = 1.2]
	\node[vertex, label=above:\strut $x_1$](x1) at (0,0){};
	\node[vertex, label=above:\strut $y_1$](y1) at (2,0){};
	
	\node[vertex, label=below:\strut $x_2$](x2) at (0,-1){};
	\node[vertex, label=below:\strut $y_2$](y2) at (2,-1){};
	
	\node[](text) at (-0.75,-0.5){$F'$};
	
	\draw[ultra thick, black](x1) to [bend left = 30] (y1);
	\draw[ultra thick, black](x1) to [bend right = 30] (y1);
	
	\draw[ultra thick, black](x2) to [bend left = 30] (y2);
	\draw[ultra thick, black](x2) to [bend right = 30] (y2);
	
	\draw[edge](x1)--(x2);
	\draw[edge](y1)--(y2);
	
\end{scope}

\end{tikzpicture}

%% file: gadgetgirth.tex
\begin{tikzpicture}

\node[](l) at (-1,2){a)};

\begin{scope}[shift={(0.1,0)}]

\begin{scope}[shift = {(3,2)}, xscale=0.8]
\node[vertex, label=below: $s$](s) at (0,0){};
\node[vertex, label=below: $t$] (t) at (3,0){};
\node[vertex,label=below: $x$](x) at (1,0){};
\node[vertex, label=below: $y$](y) at (2,0){};

\draw[ultra thick, black](s) to [bend left = 45](t);
\draw[ultra thick, black](s) to [bend right = 45](t);
\draw[edge](s)--(x);
\draw[edge](x)--(y);
\draw[edge](y)--(t);

\node[](text) at (1.5, -1){$F(s,t)$};

\end{scope}

\begin{scope}[shift = {(0,2)}, xscale = 0.8]
\node[vertex, label=below: $s$](s) at (0,0){};
\node[vertex, label=below: $t$] (t) at (3,0){};
\node[vertex,label=below: $x$](x) at (1,0){};
\node[vertex, label=below: $y$](y) at (2,0){};

\draw[ultra thick, black](s) to [bend left = 45](t);
\draw[ultra thick, black](s) to [bend right = 45](t);
\draw[edge](s)--(x);
%\draw[edge](x)--(y);
\draw[edge](y)--(t);

\node[](text) at (1.5, -1){$F$};

\end{scope}
\scriptsize

\begin{scope}[shift = {(6,2)}, scale = 2/3]
\begin{scope}[]
\node[vertex,](s) at (0,0){};
\node[](l) at (0,0.7){\strut {$s=s_1$}};
\node[vertex,] (t) at (3,0){};
\node[](l) at (3,0.7){\strut {$t_1=s_2$}};
\node[vertex,](x) at (1,0){};
\node[vertex,](y) at (2,0){};

\draw[ultra thick, black](s) to [bend left = 45](t);
\draw[ultra thick, black](s) to [bend right = 45](t);
\draw[edge](s)--(x);
\draw[edge](x)--(y);
\draw[edge](y)--(t);

\end{scope}
\begin{scope}[shift = {(3,0)}]
\node[vertex](s) at (0,0){};
\node[vertex, ] (t) at (3,0){};
\node[](l) at (3,0.7){\strut {$t_2=s_3$}};
\node[vertex,](x) at (1,0){};
\node[vertex](y) at (2,0){};

\draw[ultra thick, black](s) to [bend left = 45](t);
\draw[ultra thick, black](s) to [bend right = 45](t);
\draw[edge](s)--(x);
\draw[edge](x)--(y);
\draw[edge](y)--(t);

\end{scope}

\begin{scope}[shift = {(6,0)}]
\node[vertex](s) at (0,0){};
\node[vertex,] (t) at (3,0){};
\node[](l) at (3,0.7){\strut {$t_3=t$}};
\node[vertex](x) at (1,0){};
\node[vertex,](y) at (2,0){};

\draw[ultra thick, black](s) to [bend left = 45](t);
\draw[ultra thick, black](s) to [bend right = 45](t);
\draw[edge](s)--(x);
\draw[edge](x)--(y);
\draw[edge](y)--(t);

\end{scope}
\normalsize
\node[](text) at (4.5, -1.5){$H(s,t)$};
\end{scope}

\end{scope}

\normalsize

\node[](l) at (-1,-0.25){b)};
%\scriptsize
\begin{scope}[shift = {(0,-0.25)}, scale = 2/3]
\begin{scope}[]
\node[vertex,label=below:\strut{$s$}, ](s) at (0,0){};
\node[vertex] (t) at (3,0){};
\node[vertex,](x) at (1,0){};
\node[vertex,](y) at (2,0){};

\draw[ultra thick, black](s) to [bend left = 45](t);
\draw[ultra thick, black](s) to [bend right = 45](t);
\draw[gedge](s)--(x);
\draw[edge](x)--(y);
\draw[gedge](y)--(t);

\end{scope}
\begin{scope}[shift = {(3,0)}]
\node[vertex](s) at (0,0){};
\node[vertex, ] (t) at (3,0){};
\node[vertex,](x) at (1,0){};
\node[vertex](y) at (2,0){};

\draw[ultra thick, black](s) to [bend left = 45](t);
\draw[ultra thick, black](s) to [bend right = 45](t);
\draw[edge](s)--(x);
\draw[gedge](x)--(y);
\draw[edge](y)--(t);

\end{scope}

\begin{scope}[shift = {(6,0)}]
\node[vertex,](s) at (0,0){};
\node[vertex,label=below:\strut{$t$}, ] (t) at (3,0){};
\node[vertex](x) at (1,0){};
\node[vertex,](y) at (2,0){};

\draw[ultra thick, black](s) to [bend left = 45](t);
\draw[ultra thick, black](s) to [bend right = 45](t);
\draw[gedge](s)--(x);
\draw[edge](x)--(y);
\draw[gedge](y)--(t);

\end{scope}

\end{scope}

\begin{scope}[shift = {(6.5,-0.25)}, scale = 2/3]
\begin{scope}[]
\node[vertex,  label=below:\strut{$s$}, fill=none, dotted, thick](s) at (0,0){};
\node[vertex, ] (t) at (3,0){};
\node[vertex,](x) at (1,0){};
\node[vertex,](y) at (2,0){};

\draw[ultra thick, black](s) to [bend left = 45](t);
\draw[ultra thick, black](s) to [bend right = 45](t);
\draw[edge](s)--(x);
\draw[gedge](x)--(y);
\draw[edge](y)--(t);

\end{scope}
\begin{scope}[shift = {(3,0)}]
\node[vertex](s) at (0,0){};
\node[vertex, ] (t) at (3,0){};
\node[vertex,](x) at (1,0){};
\node[vertex](y) at (2,0){};

\draw[ultra thick, black](s) to [bend left = 45](t);
\draw[ultra thick, black](s) to [bend right = 45](t);
\draw[gedge](s)--(x);
\draw[edge](x)--(y);
\draw[gedge](y)--(t);

\end{scope}

\begin{scope}[shift = {(6,0)}]
\node[vertex](s) at (0,0){};
\node[vertex,  label=below:\strut{$t$}, fill=none, dotted, thick] (t) at (3,0){};
\node[vertex](x) at (1,0){};
\node[vertex,](y) at (2,0){};

\draw[ultra thick, black](s) to [bend left = 45](t);
\draw[ultra thick, black](s) to [bend right = 45](t);
\draw[edge](s)--(x);
\draw[gedge](x)--(y);
\draw[edge](y)--(t);

\end{scope}

\end{scope}

\end{tikzpicture}

%% file: gadget2.tex
\scalebox{1.2}{
\begin{tikzpicture}

\node[vertex, label = below:$u$](u) at (0,1) {};
\node[vertex, label =  below:$v$](v) at (1,1) {};

\draw[edge](u)--(v);
\draw[edge](u)--(-0.35,0.7);
	\draw[edge](u)--(-0.5,1);
	\draw[edge](u)--(-0.35,1.3);
	
	\draw[edge](v)--(1.4, 1.2);
	\draw[edge](v)--(1.4, 0.8);
	
	\draw [arrows = {-Stealth[reversed, reversed]}, thick](2,1)--(3,1);

\begin{scope}[shift = {(4,-0.75)}, xscale = 1.2, yscale = 1.2]
\def\k{0.8}
\footnotesize

\node[vertex, label=$v_{2i-1}$](v2j1) at (0.5,2+\k){};
\node[vertex, label=$v_{2i-3}$](v2j3) at (1.5,2+\k){};
\node[vertex, label=$v_{3}$](v3) at (2.5,2+\k){};
\node[vertex, label=$v_{1}$](v1) at (3.5,2+\k){};

\node[vertex, label=left:$v_{2i}$](v2j) at (0,2){};
\node[vertex, label=below:$v_{2i-2}$](v2j2) at (1,2){};
\node[vertex, label=below:$v_{4}$](v4) at (2,2){};
\node[vertex, label=below:$v_{2}$](v2) at (3,2){};
\node[vertex, label=above:$v$](v) at (4,2){};

\node[vertex, label=below:$u$](u) at (0,1){};
\node[vertex, label=above:$u_{2}$](u2) at (1,1){};
\node[vertex, label=above:$u_{4}$](u4) at (2,1){};
\node[vertex, label=above:$u_{2i-2}$](u2j2) at (3,1){};
\node[vertex, label=right:$u_{2i}$](u2j) at (4,1){};

\node[vertex, label=below:$u_{1}$](u1) at (0.5, 1-\k){};
\node[vertex, label=below:$u_{3}$](u3) at (1.5, 1-\k){};
\node[vertex, label=below:$u_{2i-3}$](u2j3) at (2.5, 1-\k){};
\node[vertex, label=below:$u_{2i-1}$](u2j1) at (3.5, 1-\k){};

\draw[edge](u)--(-0.35,0.7);
	\draw[edge](u)--(-0.5,1);
	\draw[edge](u)--(-0.35,1.3);
		\draw[edge](v)--(4.4, 2.2);
	\draw[edge](v)--(4.4, 1.8);

\draw[edge](v2j1)--(v2j3);
\draw[edge, dotted](v2j3)--(v3);
\draw[edge](v3)--(v1);

\draw[edge](v2j)--(v2j2);
\draw[edge, dotted](v2j2)--(v4);
\draw[edge](v4)--(v2);
\draw[edge](v2)--(v);

\draw[edge](u)--(u2);
\draw[edge](u2)--(u4);
\draw[edge, dotted](u4)--(u2j2);
\draw[edge](u2j2)--(u2j);

\draw[edge](u1)--(u3);
\draw[edge, dotted](u3)--(u2j3);
\draw[edge](u2j3)--(u2j1);

\draw[edge](v2j)--(v2j1);
\draw[edge](v2j1)--(v2j2);
\draw[edge](v2j2)--(v2j3);
\draw[edge](v4)--(v3);
\draw[edge](v3)--(v2);
\draw[edge](v2)--(v1);
\draw[edge](v1)--(v);

\draw[edge](v2j)--(u);
\draw[edge](v)--(u2j);

\draw[edge](u2j)--(u2j1);
\draw[edge](u2j1)--(u2j2);
\draw[edge](u2j2)--(u2j3);
\draw[edge](u4)--(u3);
\draw[edge](u3)--(u2);
\draw[edge](u2)--(u1);
\draw[edge](u1)--(u);
\end{scope}

\end{tikzpicture}
}

%% file: gadget.tex
\scalebox{1.2}{
\begin{tikzpicture}

\node[vertex, label = below:$u$](u) at (0,1) {};
\node[vertex, label =  below:$v$](v) at (1,1) {};

\draw[edge](u)--(v);
\draw[edge](u)--(-0.35,0.7);
	\draw[edge](u)--(-0.5,1);
	\draw[edge](u)--(-0.35,1.3);
	
	\draw[edge](v)--(1.4, 1.2);
	\draw[edge](v)--(1.4, 0.8);
	
	\draw [arrows = {-Stealth[reversed, reversed]}, thick](2,1)--(3,1);

\begin{scope}[shift = {(4,0.5)}]

\footnotesize
	\node[vertex, label = $u$](u) at (0,1) {};
	\node[vertex, label = below:$u'$](up) at (0,0) {};
	\node[vertex, label = $v_{i+1}$](vj) at (1,1) {};
	\node[vertex, label =  below:$v_{2i+2}$](v2j) at (1,0) {};
	\node[vertex, label = $v_{i}$](vj1) at (2,1) {};
	\node[vertex, label =  below:$v_{2i+1}$](v2j1) at (2,0) {};
	\node[vertex, label = $v_{i-1}$](vj2) at (3,1) {};
	\node[vertex, label =  below:$v_{2i}$](v2j2) at (3,0) {};	
	
	\node[vertex, label = $v_2$](v2) at (4,1) {};
	\node[vertex, label =  below:$v_{i+3}$](v2js) at (4,0) {};	
	\node[vertex, label = $v_1$](v1) at (5,1) {};
	\node[vertex, label =  below:$v_{i+2}$](v1j) at (5,0) {};
	\node[vertex, label = $v'$](vp) at (6,1) {};
	\node[vertex, label =  below:$v$](v) at (6,0) {};
	
	\draw[edge](u)--(up);
	\draw[edge](u)--(vj);
	\draw[edge](up)--(v2j);
	\draw[edge](vj)--(vj1);
	\draw[edge](vj)--(v2j);
	\draw[edge](v2j)--(vj1);
	\draw[edge](v2j)--(v2j1);
	\draw[edge](vj1)--(v2j1);
	\draw[edge](v1)--(vp);
	\draw[edge](v1)--(v1j);
	\draw[edge](v1j)--(vp);
	\draw[edge](v1j)--(v);
	\draw[edge](vp)--(v);
	\draw[edge](vj1)--(vj2);
	\draw[edge](v2j1)--(v2j2);
	\draw[edge](v2j1)--(vj2);
	\draw[edge](vj2)--(v2j2);
	\draw[edge](v2)--(v1);
	\draw[edge](v2)--(v2js);
	\draw[edge](v2js)--(v1);
	\draw[edge](v2js)--(v1j);

	\draw[gray,  thick, dotted](v2j2)--(v2js);
	\draw[gray,  thick, dotted](vj2)--(v2);
	
	\draw[br](6.1,-0.6)--(0.7,-0.6);

	\node[](i) at (3.5, -0.9){$i+1$ diamonds};
	
	\draw[edge](u)--(-0.35,0.7);
	\draw[edge](u)--(-0.5,1);
	\draw[edge](u)--(-0.35,1.3);
	
	\draw[edge](v)--(6.4, 0.2);
	\draw[edge](v)--(6.4, -0.2);
	\end{scope}

\end{tikzpicture}}

%% file: gadgetcoloured.tex
\begin{subfigure}[]{0.3\textwidth}
\begin{tikzpicture}
\begin{scope}[scale = 0.91]
\node[rvertex, label = $u$](u) at (0,1) {};
	\node[rvertex, label = below:$u'$](up) at (0,0) {};
	\node[rvertex, label = $v_{i+1}$](vi) at (1,1) {};
	\node[rvertex, label =  below:$v_{2i+2}$](v2i) at (1,0) {};
	\node[rvertex, label = $v_{i}$](vi1) at (2,1) {};
	\node[rvertex, label =  below:$v_{2i+1}$](v2i1) at (2,0) {};
	\node[rvertex, label = $v_1$](v1) at (3,1) {};
	\node[rvertex, label =  below:$v_{i+2}$](v1i) at (3,0) {};
	\node[rvertex, label = $v'$](vp) at (4,1) {};
	\node[rvertex, label =  below:$v$](v) at (4,0) {};
	
	\draw[medge](u)--(up);
	\draw[edge](u)--(vi);
	\draw[edge](up)--(v2i);
	\draw[edge](vi)--(vi1);
	\draw[medge](vi)--(v2i);
	\draw[edge](v2i)--(vi1);
	\draw[edge](v2i)--(v2i1);
	\draw[medge](vi1)--(v2i1);
	\draw[edge](v1)--(vp);
	\draw[medge](v1)--(v1i);
	\draw[edge](v1i)--(vp);
	\draw[edge](v1i)--(v);
	\draw[medge](vp)--(v);
	
	\draw[gray,  thick, dotted](v2i1)--(v1i);
	\draw[gray,  thick, dotted](vi1)--(v1);
	\end{scope}
	\end{tikzpicture}
	\caption{$u$ and $v$ have the same colour and $uv \in M$.}
	\end{subfigure}
	\hfill
	\begin{subfigure}[]{0.3\textwidth}
	\begin{tikzpicture}
	\begin{scope}[scale = 0.91]
\node[rvertex, label = $u$](u) at (0,1) {};
	\node[rvertex, label = below:$u'$](up) at (0,0) {};
	\node[rvertex, label = $v_{i+1}$](vi) at (1,1) {};
	\node[rvertex, label =  below:$v_{2i+2}$](v2i) at (1,0) {};
	\node[rvertex, label = $v_{i}$](vi1) at (2,1) {};
	\node[rvertex, label =  below:$v_{2i+1}$](v2i1) at (2,0) {};
	\node[rvertex, label = $v_1$](v1) at (3,1) {};
	\node[rvertex, label =  below:$v_{i+2}$](v1i) at (3,0) {};
	\node[rvertex, label = $v'$](vp) at (4,1) {};
	\node[rvertex, label =  below:$v$](v) at (4,0) {};
	
	\draw[edge](u)--(up);
	\draw[edge](u)--(vi);
	\draw[medge](up)--(v2i);
	\draw[medge](vi)--(vi1);
	\draw[edge](vi)--(v2i);
	\draw[edge](v2i)--(vi1);
	\draw[edge](v2i)--(v2i1);
	\draw[edge](vi1)--(v2i1);
	\draw[medge](v1)--(vp);
	\draw[edge](v1)--(v1i);
	\draw[edge](v1i)--(vp);
	\draw[edge](v1i)--(v);
	\draw[edge](vp)--(v);
	
	\draw[gray,  thick, dotted](v2i1)--(v1i);
	\draw[gray,  thick, dotted](vi1)--(v1);
	\end{scope}
	\end{tikzpicture}
	\caption{$u$ and $v$ have the same colour and $uv \not\in M$.}
	\end{subfigure}
	\hfill
\begin{subfigure}[]{0.3\textwidth}

\begin{tikzpicture}
\begin{scope}[scale = 0.91]
	\node[rvertex, label = $u$](u) at (0,1) {};
	\node[rvertex, label = below:$u'$](up) at (0,0) {};
	\node[bvertex, label = $v_{i+1}$](vi) at (1,1) {};
	\node[bvertex, label =  below:$v_{2i+2}$](v2i) at (1,0) {};
	\node[bvertex, label = $v_{i}$](vi1) at (2,1) {};
	\node[bvertex, label =  below:$v_{2i+1}$](v2i1) at (2,0) {};
	\node[bvertex, label = $v_1$](v1) at (3,1) {};
	\node[bvertex, label =  below:$v_{i+2}$](v1i) at (3,0) {};
	\node[bvertex, label = $v'$](vp) at (4,1) {};
	\node[bvertex, label =  below:$v$](v) at (4,0) {};
	
	\draw[edge](u)--(up);
	\draw[medge](u)--(vi);
	\draw[medge](up)--(v2i);
	\draw[edge](vi)--(vi1);
	\draw[edge](vi)--(v2i);
	\draw[edge](v2i)--(vi1);
	\draw[edge](v2i)--(v2i1);
	\draw[medge](vi1)--(v2i1);
	\draw[edge](v1)--(vp);
	\draw[medge](v1)--(v1i);
	\draw[edge](v1i)--(vp);
	\draw[edge](v1i)--(v);
	\draw[medge](vp)--(v);
	
	\draw[gray,  thick, dotted](v2i1)--(v1i);
	\draw[gray,  thick, dotted](vi1)--(v1);
	\end{scope}
	\end{tikzpicture}
	\caption{$u$ and $v$ have different colours.}
	\end{subfigure}